\documentclass{article}

\usepackage{amssymb} 
\usepackage{latexsym} 
\usepackage{amsmath}
\usepackage{mathtools} 
\usepackage{tabls} 
\usepackage{graphicx}
\usepackage{todonotes}
\usepackage{overpic}
\usepackage{subcaption}

\usepackage{geometry}

\usepackage[utf8]{inputenc}
\usepackage{color}
\usepackage{epsfig}
\usepackage{graphicx}
\usepackage[all]{xy}
\input xy
\usepackage[mathcal]{eucal}
\usepackage{enumerate}
\usepackage{enumitem}
\usepackage{hyperref}
\usepackage{accents}

\newcommand{\subscript}[1]{{\rm S$_ #1$}}

\newtheorem{theorem}{Theorem}
\newtheorem{definition}{Definition}
\newtheorem{remark}{Remark}
\newtheorem{prop}{Proposition}
\newtheorem{lemma}{Lemma}
\newtheorem{corollary}{Corollary}

\newenvironment{proof}{\noindent{\bf Proof.}}%
{\hspace*{\fill}$\Box$\par\vspace{4mm}}

\newenvironment{claimproof}[2]{\par\noindent\underline{Proof of #1:}\space#2}{\vspace*{2mm}}

\def\cadre{$$\vcenter\bgroup\advance\hsize by -2em\noindent
	\refstepcounter{equation}(\theequation)~\ignorespaces}
\makeatletter
\def\endcadre{\egroup\eqno$$\global\@ignoretrue}

\newcommand{\Upper}{\text{Upper}}
\newcommand{\Ext}{\text{Cross}}

\begin{document}
\pagestyle{plain}

\title{Two New Characterizations of Path Graphs}
\date{}
%  AUTORI NON LNCS
\author{Nicola Apollonio\footnote{Istituto per le Applicazioni del
		Calcolo, M. Picone, Via dei Taurini 19, 00185 Roma, Italy.
		\texttt{nicola.apollonio@cnr.it}} \and
	{Lorenzo Balzotti\footnote{Dipartimento di Scienze di Base e Applicate per l’Ingegneria, Sapienza Universit\`a di Roma, Via Antonio Scarpa, 16,
00161 Roma, Italy. \texttt{lorenzo.balzotti@sbai.uniroma1.it}.}}
}

\maketitle

\begin{abstract} 
Path graphs are intersection graphs of paths in a tree. We start from the characterization of path graphs by Monma and Wei [C.L.~Monma,~and~V.K.~Wei, Intersection {G}raphs of {P}aths in a {T}ree, J. Combin. Theory Ser. B, 41:2 (1986) 141--181] and we reduce it to some 2-colorings subproblems, obtaining the first characterization that directly leads to a polynomial recognition algorithm. Then we introduce the collection of the \emph{attachedness graphs} of a graph and we exhibit a list of minimal forbidden 2-edge colored subgraphs in each of the \emph{attachedness graph}.
\end{abstract}


\section{Introduction}\label{introduction}

A path graph is the intersection graph of paths in a tree. Other variants of the Path/Tree intersection model are obtained by requiring edge-intersection (or even arc intersection) and by specializing the shape of $T$ (e.g.: directed, rooted). The class of path graphs is clearly closed under taking induced subgraphs. 
Path graphs were introduced by Renz~\cite{renz} who also posed the question of characterizing them by forbidden subgraphs giving at the same a first partial answer.~The question has been fully answered only recently by L\'{e}v\^{e}que, Maffray and Preissmann~\cite{bfm}. 	

Path graph is a class of graphs between \emph{interval graphs} and \emph{chordal graphs}. A graph is a chordal graph if it does not contain a \emph{hole} as an induced subgraph, where a hole is a chordless cycle of length at least four. By specializing the celebrated characterization of chordal graphs due to Gavril~\cite{gavril1}, still Gavril gave the first characterization of path graphs~\cite{gavril_UV_algorithm}. A graph is an interval graph if it is the intersection graph of a family of intervals on the real line; or, equivalently, the intersection graph of a family of subpaths of a path. Interval graphs were characterized by Lekkerkerker and Boland~\cite{lekkerkerker-boland} as chordal graphs with no \emph{asteroidal triples}, where an asteroidal triple is a stable set of three vertices such that each pair is connected by a path avoiding the neighborhood of the third vertex. A generalization of the steroidal triples is introduced in~\cite{mouatadid-robere} where path graphs are characterized by forbidding \emph{sun systems}.

Inspired by the work of Tarjan~\cite{tarjan}, Monma and Wei~\cite{mew}, presented a general framework to recognize and realize intersection graphs having as intersection model all possible variants of the Path/Tree model. In particular, they characterized chordal graphs, path graph, \emph{directed path graphs} and \emph{rooted directed path graphs}, where the latters are variants of path graphs. A graph is a directed path graphs if it is the intersection graph of a family of paths of a directed tree. Directed path graphs were characterized first by Panda~\cite{panda} by a list of forbidden induced subgraphs and then by Cameron, Ho\'{a}ng and L\'{e}v\^{e}que~\cite{cameron-hoang_1,cameron-hoang_2} by extending the concept of asteroidal triples. A graph is a rooted path graphs if it is the intersection graph of a family of paths of a rooted directed tree. In the actual state of the art, there is no characterization of rooted path graphs by forbidden subgraphs or by concepts similar to asteroidal triples. The characterizations of these graphs' classes in~\cite{mew} also describe directly polynomial recognition algorithms for chordal graphs and directed path graphs but not for rooted path graphs and path graphs.

The following strict inclusions between introduced graphs' classes are proved in~\cite{mew}:
\begin{equation*}
\text{interval graphs $\subset$ rooted path graphs $\subset$ directed path graphs $\subset$ path graphs $\subset$ chordal graphs}.
\end{equation*}
The first recognition algorithm for path graphs was given by Gavril~\cite{gavril_UV_algorithm}, and it has $O(n^4)$ worst-case time complexity, where the input graph has $n$ vertices and $m$ edges. The fastest algorithms are due to Sch\"{a}ffer~\cite{schaffer} and Chaplick~\cite{chaplick} and both have $O(mn)$ worst-case time complexity. The first is a sophisticated backtracking algorithm based on characterization of Monma and Wei, while the second uses PQR-trees. Another algorithm is proposed in~\cite{dah} and claimed to run in $O(n+m)$ time, but is not considered
correct (see comments in [\cite{chaplick}, Section 2.1.4]).

Gavril also gave the first algorithm to recognize directed path graph~\cite{gavril_DV_algorithm}. Chaplick \emph{et al.}~\cite{chaplick-gutierrez} describe a linear algorithm able to say if a path graph is a directed path graph, by assuming to have the realization of path graph as the intersection of a family of paths of a tree. This implies that algorithms in~\cite{chaplick,schaffer} can be extended to recognition algorithms for directed path graphs with the same time complexity. At the state of art, these are the two fastest algorithms.

\paragraph{Our Contribution} 
Our first characterization of path graphs is the unique, at the state of the art, that directly implies a polynomial recognition algorithm, and we obtain it by starting from Monma and Wei's characterization~\cite{mew}. The second characterization follows from the first one and consists in a list of minimal forbidden subgraphs and a list of minimal forbidden induced subgraphs on a graph derived from the input graph. Now we describe in detail the two new characterizations.

In the characterization by Monma and Wei~\cite{mew} the graph is decomposed recursively by \emph{clique separators} and in every decomposition step one has to solve a coloring problem (see Theorem~\ref{thm:mw} and Section~\ref{section:mew_characterization} for all definitions and notations); the solution of the coloring problem is used to represent the graph as the intersection graph of a family of paths in a tree. The difficulty with their coloring problem led them to not prove that it can be solved in polynomial time. In our first characterization we simplify Monma and Wei's characterization by reducing it to some 2-coloring subproblems that are clearly solvable in polynomial time (see Section~\ref{section:weak_coloring}, in particular Subsection~\ref{sub:2-coloring}). Thus this characterization also describes a polynomial recognition algorithm. Moreover, it has two consequences.

%$\bullet$ We answer to the open problem by Monma and Wei~\cite{mew} asking for a polynomial algorithm which arises from their characterization (see Subsection~\ref{sub:2-coloring}).

$\bullet$ The obstructions to our 2-coloring subproblems are well-known, this allows us to give our second characterization, which consists in a list of minimal forbidden subgraphs and a list of minimal forbidden induced subgraphs of the \emph{attachedness graphs} of the input graphs (see Section~\ref{section:forbidden_subgraphs}, in particular Theorem~\ref{cor:all}).% (in Subsection~\ref{sub:LMP} we compare our list with the list of forbidden subgraphs in~\cite{bfm}).

$\bullet$ Our first characterization is used in~\cite{balzotti} to describe a recognition algorithm that specializes for path graphs and directed path graphs. The algorithm is based on Theorem~\ref{th:characterizazion_1} and the partition given in Equation~\eqref{dstorto}. Even if the worst-case time complexities are not improved, at the state of art, the recognition algorithm for directed path graphs in~\cite{balzotti} is the unique that does not use the results in~\cite{chaplick-gutierrez}, in which a linear algorithm is given that is able to establish whether a path graph is a directed path graph too. Moreover, the recognition algorithm for path graph in~\cite{balzotti} has an easier and more intuitive implementation than Sch\"{a}ffer's backtracking algorithm~\cite{schaffer} and it requires no complex data structures while algorithm in~\cite{chaplick} is built on PQR-trees. In this way, our characterization allowed us to simplify and unify the recognition algorithms for path graphs and directed path graphs. We refer to~\cite{balzotti} for further details.

%In our first characterization, we simplify their coloring problem (see Section~\ref{section:weak_coloring}) by reducing it to some 2-coloring problems (see Definition~\ref{def:weak_colorability} and Theorem~\ref{th:characterizazion_1}). From one hand, this simplification implies a new polynomial-time algorithm for recognizing/realizing path graphs; from the other hand it allows us to exhibit all forbidden configurations to the property of being a path graph in the form of forbidden 2-colored subgraphs of the \emph{attachedness graphs} of the input graphs (Section~\ref{section:forbidden_subgraphs}). This new set of forbidden subgraph leads to our second characterization. Main results are explained in Theorem~\ref{th:characterizazion_1} and in Theorem~\ref{cor:all}.

%We note that this paper is the ``theorical part'' of a wider study of path graphs and directed path graphs. The ``algorithmic part'' is described in~\cite{balzotti} where two algorithms to recognize path graphs and directed path graphs are presented. More details are given in Section~\ref{sec:consequences}.

\paragraph{Organization}
The paper is organized as follows. In Section~\ref{section:mew_characterization} we give a detailed discussion of Monma and Wei's characterization of path graphs~\cite{mew} and we define the basic concepts (e.g.,
the notion of \emph{attachedness graphs}). Section~\ref{section:weak_coloring} is devoted to our first characterization (summarized in Theorem~\ref{th:characterizazion_1}) that consists in a simplification of Monma and Wei's one. Then, in Section~\ref{section:forbidden_subgraphs}, we use the results to characterize paths graphs by a list of forbidden subgraphs in their attachedness graphs (see Theorem~\ref{cor:all}). In Section~\ref{sec:conclusions} the conclusions are given.

\section{Monma and Wei's characterization of path graphs}\label{section:mew_characterization}

In this section we show Monma and Wei's characterization of path graphs (Theorem~\ref{thm:mw}) that is based on Gavril's one (Theorem~\ref{thm:gavril2}). Let's start with a formal definition of path graphs.

A graph $G$ is a \emph{path graph} if there is a tree $T$ (\emph{the host tree of $G$}), a collection $\mathcal{P}$ of paths of $T$ and a bijection $\phi: V(G)\rightarrow \mathcal{P}$ such that two vertices $u$ and $v$ of $G$ are adjacent in $G$ if and only if the vertex-sets of paths $\phi(u)$ and $\phi(v)$ intersect. In Figure~\ref{fig:path_graph} there is a path graph $G$ on the left and the host tree of $G$ with a collection of paths $\mathcal{P}$ that realizes $G$ on the right.

\begin{figure}[h]
\centering
\begin{overpic}[width=14cm,percent]{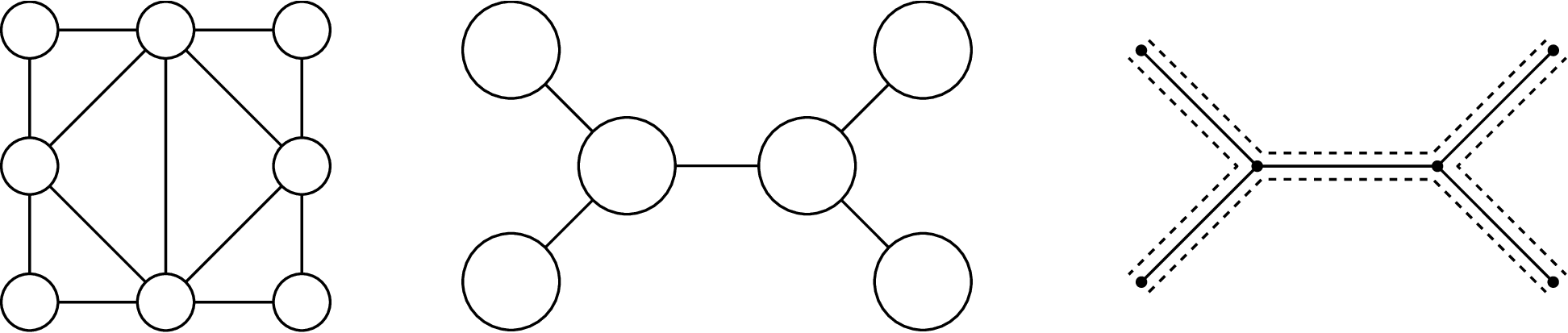}         

\put(1.2,18.8){$a$}
\put(10,18.8){$b$}
\put(1.2,10.2){$c$}
\put(1.2,1.4){$d$}
\put(10,1.4){$e$}
\put(18.5,18.8){$f$}
\put(18.55,10.3){$g$}
\put(18.5,1.4){$h$}

\put(30.3,17.5){$a\,b\,c$}
\put(30.3,2.7){$c\,d\,e$}
\put(38,10.1){$b\,c\,e$}
\put(49.3,10.1){$b\,e\,g$}
\put(56.5,17.5){$b\,f\,g$}
\put(56.5,2.7){$e\,g\,h$}

\put(70.2,19){$p_a$}
\put(85,12.5){$p_b$}
\put(75,10.2){$p_c$}
\put(70.2,2){$p_d$}
\put(85,8.2){$p_e$}
\put(99.7,19){$p_f$}
\put(95,10.2){$p_g$}
\put(99.7,2){$p_h$}
\end{overpic}
\caption{on the left a path graph $G$, in the center the clique path tree of $G$, on the right the host tree of $G$ and the collection $\mathcal{P}=\{p_a,\ldots,p_h\}$ that realizes $G$. Note that $p_a,p_d,p_f,p_h$ are composed by only one vertex.}
\label{fig:path_graph}
\end{figure}

Path graphs were first characterized by Gavril~\cite{gavril_UV_algorithm} through the notion of \emph{clique path tree} as follows (unless otherwise stated, maximal cliques are referred to as cliques, where a \emph{clique} is a set of pairwise adjacent vertices).

\begin{theorem}[Gavril~\cite{gavril_UV_algorithm}]\label{thm:gavril} A graph $G$ is a path graph if and only if it possesses a \emph{clique path tree}, namely, a tree $T$ whose vertices are the cliques of $G$ with the property that the set of cliques of $G$ containing a given vertex $v$ of $G$ induces a path in $T$. 
\end{theorem}
Theorem~\ref{thm:gavril2} specializes the celebrated characterization of chordal graphs, also due to Gavril~\cite{gavril1}, as those graphs possessing a \emph{clique tree} (equivalently, as the intersection graphs of a collection of subtrees in a given tree) as stated below.
\begin{theorem}[Gavril~\cite{gavril1}]\label{thm:gavril2}
A graph $G$ is a chordal graph if and only if it possesses a \emph{clique tree}, namely, a tree $T$ on the set of cliques of $G$ with the property that the set of cliques of $G$ containing a given vertex $v$ of $G$ induces a subtree in $T$.
\end{theorem}
Notice that since a clique path tree is a particular clique tree, Theorem~\ref{thm:gavril2} also implies that path graphs are chordal. In Figure~\ref{fig:path_graph}, in the center there is the clique path tree of the path graph on the left.

A clique $Q$ is a \emph{clique separator} if the removal of $Q$ from $G$ disconnects $G$ into more than one connected component (without loss of generality, throughout the paper, we suppose that $G$ is connected). If a graph $G$ has no clique separator, then $G$ is called \emph{atom}. In~\cite{mew} it is proved that an atom is a path graph if and only if it is a chordal graph.

Given a clique separator $Q$ of a graph $G$ let $G-Q$ have $s$ connected components, $s\geq 2$ with vertex-sets $V_1,\ldots,V_s$, respectively. We define $\gamma_i=G[V_i\cup Q]$, $i=1,\ldots,s$ and $\Gamma_Q=\{\gamma_1,\ldots,\gamma_s\}$. 
%Set $\Gamma_Q$ will be referred to as a \emph{$Q$-presentation}\todob{$Q$-presentation e' inutile e $Q$ clique separator e' gia' detto} of $G$. Clique $Q$ is a \emph{clique separator}. 
A clique $K$ of a subgraph $\gamma$ of $\Gamma_Q$ is called a  \emph{relevant clique}, if $K \cap Q\not=\emptyset$ and $K\neq Q$. A \emph{neighboring subgraph} of a vertex $v\in V(G)$ is a member $\gamma\in \Gamma_Q$ such that $v$ belongs to some relevant clique $K$ of $\gamma$. For instance, in Figure~\ref{fig:example1} referring to the graph on the left, all the $\gamma_i$'s but $\gamma_5$ are neighboring subgraphs of the vertex in the north-east corner of the clique separator $Q$, while all the $\gamma_i$'s but $\gamma_2$ and $\gamma_3$ are neighboring subgraphs of the vertex in the south-west corner of $Q$. We say that two subgraphs $\gamma$ and $\gamma'$ are \emph{neighbouring} if they are neighbouring subgraphs of some vertex $v\in Q$; a subset $W\subseteq \Gamma_Q$ whose elements are neighbouring subgraphs will be referred to as a \emph{neighbouring set} (e.g, \emph{neighbouring pairs}, \emph{neighbouring triples} etc). Monma and Wei~\cite{mew}, defined the following binary relations on $\Gamma_Q$.

%Given a clique separator $Q$ of a graph $G$ let $G-Q$ have $s$ connected components, $s\geq 2$ with vertex-sets $V_1,\ldots,V_s$, respectively. We define $\gamma_i=G[V_i\cup Q]$, $i=1,\ldots,s$ and $\Gamma_Q=\{\gamma_1,\ldots,\gamma_s\}$, where for a subset $A$ of $V(G)$, we denote the graph induced by $A$ in $G$ by $G[A]$. A clique of a member $\gamma$ of $\Gamma_Q$ is called a  \emph{relevant clique}, if $K \cap Q\not=\emptyset$ and $K\neq Q$. A \emph{neighboring subgraph} of a vertex $v\in V(G)$ is a member $\gamma\in \Gamma_Q$ such that $v$ belongs to some relevant clique $K$ of $\gamma$. For instance, in Figure~\ref{fig:example1} referring to the graph on the left, all the $\gamma_i$'s but $\gamma_5$ are neighboring subgraphs of the vertex in the north-east corner of the clique separator $Q$, while all the $\gamma_i$'s but $\gamma_2$ and $\gamma_3$ are neighboring subgraphs of the vertex in the south-west corner of $Q$. Monma and Wei~\cite{mew}, defined the following binary relations on $\Gamma_Q$.

\begin{figure}[h]
\centering
\begin{overpic}[width=14cm]{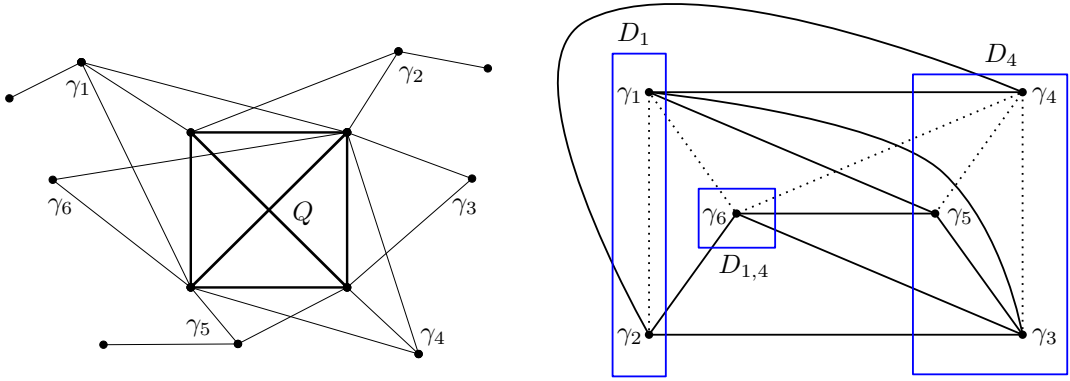}    
\put(5.7,27.2){$\gamma_1$} 
\put(37,28.5){$\gamma_2$} 
\put(42,16){$\gamma_3$} 
\put(39,3.7){$\gamma_4$} 
\put(17,4.5){$\gamma_5$} 
\put(4,16){$\gamma_6$}    
\put(27,15){$Q$}
\put(57.5,26.5){$\gamma_1$} 
\put(57.5,3.7){$\gamma_2$} 
\put(96.5,3.7){$\gamma_3$} 
\put(96.5,26.5){$\gamma_4$} 
\put(88.5,15){$\gamma_5$} 
\put(65.5,15){$\gamma_6$}    
\put(57.5,32){$D_1$} 
\put(67,10){$D_{1,4}$} 
\put(92,30){$D_4$}        
\end{overpic}
     \caption{A graph $G$ (on the right) and the $Q$-attachedness graph of $G$ (on the left). The sets displayed by a colored boundary are defined in \eqref{dgamma} and \eqref{dgammagamma'}. Remark that graph G is not a path graph.}
\label{fig:example1}
\end{figure}

\begin{description}
\item\emph{Attachedness}, denoted by $\Join$ and defined by $\gamma\Join \gamma'$ if and only if there is a relevant clique $K$ of $\gamma$ and a relevant clique $K'$ of $\gamma'$ such that $K\cap K'\cap Q\neq\emptyset$. In particular, $\gamma$ and $\gamma'$ are neighboring subgraphs of each vertex $v\in K\cap K'\cap Q$. 

\item\emph{Dominance}, denoted by $\leq$ and defined by $\gamma\leq \gamma'$ if and only if $\gamma\Join \gamma'$ and for each relevant clique $K'$ of $\gamma'$ either $K\cap Q\subseteq K'\cap Q$ for each relevant clique $K$ of $\gamma$ or $K\cap K'\cap Q=\emptyset$ for each relevant clique $K$ of $\gamma$. In  Figure~\ref{fig:example1}, graph of the right, pairs of $\leq$-comparable subgraphs of the graph of the left are joined by a dotted edge.

\item\emph{Antipodality}, denoted by $\leftrightarrow$ and defined by $\gamma \leftrightarrow\gamma'$ if and only if there are relevant cliques $K$ of $\gamma$ and $K'$ of $\gamma'$ such that $K\cap K'\cap Q\not=\emptyset$ and $K\cap Q$ and $K'\cap Q$ are inclusion-wise incomparable. In  Figure~\ref{fig:example1}, graph of the right, pairs of antipodal subgraphs of the graph of the left are joined by a solid edge.
\end{description}

Antipodality and dominance relations are disjoint binary relations on $\Gamma_Q$ whose union is the relation $\Join$. Therefore $(\gamma\leq \gamma'$, $\gamma'\leq \gamma$ or $\gamma\leftrightarrow \gamma')$ if and only if $(\gamma\Join\gamma')$. Both $\Join$ and $\leftrightarrow$ are symmetric and only $\leftrightarrow$ is irreflexive. Hence, after neglecting reflexive pairs, $(\Gamma_Q,\leftrightarrow)$,  $(\Gamma_Q,\Join)$ are simple undirected graphs on $\Gamma_Q$ referred to as, respectively, the \emph{$Q$-antipodality}, and the \emph{$Q$-attachedness graph of $G$}. The edges of the $Q$-antipodality graph of $G$ are called \emph{antipodal edges} while those edges of the $Q$-attachedness graph of $G$ which are not antipodal edges, are called \emph{dominance edges}. The \emph{$Q$-dominance} of $G$ is the graph on $\Gamma_Q$ having as edges the dominance edges (i.e., the complement of $(\Gamma_Q,\leftrightarrow)$ in $(\Gamma_Q,\Join)$). Hence the edge-sets of the $Q$-antipodality and the $Q$-dominance graphs of $G$ partition the edge-set of the $Q$-attachedness graph of $G$ and the latter is naturally 2-edge colored by the antipodality edges and by the dominance edges. We adopt the pictorial convention to represent antipodality edges by thin lines and dominance edges by dotted lines.
 
To understand the following definition, for $n\in \mathbb{N}$, we denote by $[n]$ the interval $\{1,2,\ldots,n\}$. Moreover, if $f$ is a map between sets $A$ and $B$ and $X\subseteq A$, then $f(X)$ is the image of $X$ under $f$, namely, $f(X)=\{f(x) \ |\ x\in X\}$.

\begin{definition}\label{def:strong_colorability}
Let $Q$ be a clique separator of $G$, we say that $G$ is \emph{strong $Q$-colorable} if there exists $f:\Gamma_Q\rightarrow [s]$ such that:
\begin{enumerate}[label=\thedefinition.(\arabic*), ref=\thedefinition.(\arabic*)]\itemsep0em
\item\label{com:mw_i} if $\gamma\leftrightarrow \gamma'$, then $f(\gamma)\not=f(\gamma')$;
\item\label{com:mw_ii} if $\{\gamma,\gamma'\gamma''\}$ is neighboring triple, then $|f(\{\gamma,\gamma',\gamma''\})|\leq 2$.
\end{enumerate}
\end{definition}

We refer to a coloring $f$ satisfying the conditions of Definition~\ref{def:strong_colorability} as a \emph{strong $Q$-coloring}. We use the term ``strong'' because in Section~\ref{section:weak_coloring} we introduce a weaker coloring and we prove that they are equivalent.

\begin{theorem}[Monma and Wei~\cite{mew}]\label{thm:mw}
A chordal graph $G$ is a path graph if and only if $G$ is an atom or for a clique separator $Q$ each graph $\gamma\in \Gamma_Q$ is a path graph and $G$ is strong $Q$-colorable.
\end{theorem}

Now we explain the conditions of Definition~\ref{def:strong_colorability} in few words. Let $T$ be a clique path tree of $G$. The removal of clique separator $Q$ from $G$ disconnects $G$ in more connected components, but it also disconnects the $T$ in more subtrees. In a way, the coloring $f$ associates a connected components to the subtrees. The first condition implies that two antipodal connected components $\gamma$ and $\gamma'$ need to be in two distinct subtrees, indeed, if not, then for some $v\in (V(\gamma)\cap Q)\setminus V(\gamma')$ or $v\in(V(\gamma')\cap Q)\setminus V(\gamma)$ the set of clique of $G$ that contains $v$ does not induce a connected path in $T$. The second condition says that all connected components that contain $v$ need to be in at most two distinct subtrees, indeed, if not, the set of clique of $G$ that contains $v$ does not induce a path in $T$.

In the following corollary, we translate Theorem~\ref{thm:mw} from a recursive fashion to a local one that is more useful to our purposes. We recall that a graph with no clique separator (i.e., an atom) is a path graph if and only if it is chordal.

\begin{corollary}\label{cor:local_mew}
A chordal graph $G$ is a path graph if and only if $G$ is strong $Q$-colorable, for all  clique separator $Q$ of $G$.
\end{corollary}

By Corollary~\ref{cor:local_mew}, deciding whether a graph $G$ is a path graph is tantamount to deciding whether $G$ is strong $Q$-colorable for each separator $Q$. It is thus natural to wonder whether there are obstructions to strong $Q$-colorability and, if so, how do such obstructions look like on the attachedness graph of $G$. One such obstruction is easily recognized (see~\cite{mew}): let $\{\gamma,\gamma',\gamma''\}\subseteq \Gamma_Q$ be a neighboring triple and suppose that $\gamma$, $\gamma'$ and $\gamma''$ are pairwise antipodal (hence $\{\gamma,\gamma',\gamma''\}$ induces a triangle on the $Q$-antipodal graph of $G$). We refer to any such triple to as a \emph{full antipodal triple}. It is clear that if $\Gamma_Q$ contains a full antipodal triple, then $G$  is not strong $Q$-colorable because the two conditions in Definition~\ref{def:strong_colorability} cannot be both satisfied. 
For later reference we formalize this easy fact in a lemma.
\begin{lemma}\label{lemma:nofull}
Let $Q$ be a clique separator of $G$. If $G$ is strong $Q$-colorable, then $\Gamma_Q$ has no full antipodal triple.
\end{lemma}

\section{A \emph{weak} coloring}\label{section:weak_coloring}

In this section we introduce a \emph{weak coloring} that is used in Theorem~\ref{th:characterizazion_1} to give our first characterization of path graphs. This characterization simplifies Monma and Wei's one~\cite{mew} by an algorithmic point of view justifying the terms ``strong'' and ``weak''. This simplification is explained in Subsection~\ref{sub:2-coloring}. In Subsection~\ref{sub:proof_characterization_1} we prove Theorem~\ref{th:characterizazion_1}.

Dominance is a reflexive and transitive relation. Hence $(\Gamma_Q,\leq)$ is a preorder. We assume that such a preorder is in fact a partial order. The latter assumption is not restrictive as showed implicitly in Sch\"{a}ffer~\cite{schaffer} and explicitly as follows. Let $\sim$ be the equivalence relation on $\Gamma_{Q}$ defined by  $\gamma\sim\gamma'\Leftrightarrow\left(\gamma\leq\gamma'\wedge \gamma'\leq\gamma\right)$, namely, $\sim$ is the standard equivalence relation associated with a preorder. If $\gamma\sim \gamma'$ for some two $\gamma,\,\gamma'\in \Gamma_Q$, then for any $\gamma''\in \Gamma_Q$ it holds that $\gamma\leftrightarrow \gamma''$ if and only if $\gamma'\leftrightarrow \gamma''$. Analogously, if $\gamma\sim \gamma'$ for some two $\gamma,\,\gamma'\in \Gamma_Q$, then for any $\gamma''\in \Gamma_Q$ it holds that $\gamma\leq \gamma''$ if and only if $\gamma'\leq \gamma''$. Moreover, if $\gamma$ is a neighboring of $v$, for some $v\in Q$, and $\gamma\sim\gamma'$, then $\gamma'$ is a neighboring of $v$.

The following lemma shows that it is not restrictive to assume that $\Gamma_Q=\Gamma_Q/\sim$.

\begin{lemma}\label{lemma:cong}
Let $Q$ be a clique separator of $G$. If there exists $f:\Gamma_{Q}/\sim\to[m]$ satisfying~\ref{com:mw_i} and~\ref{com:mw_ii}, then $G$ is strong $Q$-colorable.
\end{lemma}
\begin{proof}
Let $\widetilde{f}:\Gamma_{Q}\rightarrow[m]$ be defined by $\widetilde{f}(\gamma)=f([\gamma]_{\sim})$. It holds that $\widetilde{f}$ satisfies~\ref{com:mw_i} and~\ref{com:mw_ii}, hence $\widetilde{f}$ is a strong $Q$-coloring.
\end{proof}

After the lemma, we assume that $(\Gamma_Q,\leq)$ is a partial order for every clique separator $Q$ of $G$. In other words, we assume $\Gamma_Q=\Gamma_Q/\sim$.

Given a clique separator $Q$ of $G$, we define $\Upper_Q=\{u\in\Gamma_Q \,|\, u\not\leq\gamma, \text{ for all } \gamma\in\Gamma_Q\}$ the set of \emph{upper bounds} of $\Gamma_Q$ with respect to $\leq$. 
%Given $\gamma,\gamma',\gamma''\in\Gamma_Q$, we say that $\{\gamma,\gamma',\gamma''\}$ is a \emph{full antipodal triangle} if $\gamma,\gamma',\gamma''$ are pairwise antipodal and $\gamma,\gamma',\gamma''$ are neighboring of $v$, for some $v\in Q$.
From now on we fix $(u_1,u_2,\ldots,u_\ell)$ an ordering of $\Upper_Q$ and for all $i,j\in[\ell]$ and $i<j$ we define

\begin{equation}\label{dgamma}	
D_{i}^Q=\{\gamma\in \Gamma_Q \ |\ \gamma\leq u_i  \text{ and } \gamma\nleq u_j,\,\,\forall j\in[\ell]\setminus \{i\}\},
\end{equation}
\begin{equation}\label{dgammagamma'}
D_{i,j}^Q=\{\gamma\in\Gamma_Q \ |\ \gamma\leq u_i,\gamma\leq u_j  \text{ and } \gamma\nleq u_k,\,\,\forall k\in[\ell]\setminus \{i,j\}\},
\end{equation}	
\begin{equation}\label{dstorto}
\mathcal{D}^Q=\Big\{D_i^Q \,|\, i\in[\ell]\Big\}\cup\Big\{D_{i,j}^Q \,|\, i,j\in[\ell],i<j\Big\}.
\end{equation}	

In few words, $D_i^Q$ consists of the elements of $\Gamma_Q$ dominated only by $u_i$ and no other upper bound, while $D_{i,j}^Q$ consists of those elements of $\Gamma_Q$ dominated only by $u_i$ and $u_j$ and no other upper bound. Referring to Figure~\ref{fig:example1}, $\Upper_Q=\{\gamma_1,\gamma_4\}$, if we fix the ordering $(u_1,u_2)=(\gamma_1,\gamma_4)$, then $D_1^Q=\{\gamma_1,\gamma_2\}$, $D_2^Q=\{\gamma_3,\gamma_4,\gamma_5\}$ and $D_{1,2}^Q=\{\gamma_6\}$. If no confusions arise, we omit the superscript $Q$. 

\begin{remark}\label{remark:D_i_is_a_partition}
If $\Upper_Q$ has no full antipodal triple, then for every $\gamma\in\Gamma_Q$ there are at most two different $u,u'\in \Upper_Q$ such that $\gamma\leq u$ and $\gamma\leq u'$. Thus  the $D_i$'s and the  $D_{i,j}$'s form a partition of $\Gamma_Q$.
\end{remark}

Before giving the definition of weak coloring we need some preliminary definitions.
%Given a clique separator $Q$ of $G$, we present the \emph{weak $Q$-coloring} (see Definition~\ref{def:weak_colorability}) and then we prove that $G$ is strong $Q$-colorable if and only if it is weak $Q$-colorable and $\Gamma_Q$ has no full antipodal triple (see Theorem~\ref{th:characterizazion_1}). At the end of this section, we explain why our coloring is easier than the coloring proposed by Monma and Wei, justifying the terms ``strong'' and ``weak''.
%The weak $Q$-coloring is based on the partition of elements of $\Gamma_Q$ in the $D_i$'s and the $D_{i,j}$'s. We build the definition of weak $Q$-coloring in to steps. First we introduce the preliminary definition of \emph{partial $Q$-coloring}. 
Let $\Ext_Q=\{\gamma\in\Gamma_Q \,|\,$ $\gamma\in D$, for some $D\in\mathcal{D}$, and $\gamma\leftrightarrow\gamma'$, for some $\gamma'\not\in D$\}. In few words, $\Ext_Q$ is composed by all elements in $D$, varying $D\in\mathcal{D}$, that are antipodal to at least one element not in $D$, i.e., $\Ext_Q$ is composed by all elements that ``cross'' the partition through antipodality.

\begin{definition}\label{def:partial_coloring}
Let $Q$ be a clique separator of $G$. Let $(u_1,u_2,\ldots,u_\ell)$ be any ordering of $\Upper_Q$. 
We say that $f:\Upper_Q\cup\Ext_Q\rightarrow [\ell]$ is a \emph{partial $Q$-coloring} if $f$ satisfies the following: 
\begin{enumerate}[label=\thedefinition.(\arabic*), ref=\thedefinition.(\arabic*)]\itemsep0em
\item\label{item:col_2} for all $i\in[\ell]$, $f(u_i)=i$, 
\item\label{item:col_5} for all $i\in[\ell]$, for all $\gamma\in D_i$, if $\exists \gamma\not\in D_i$ such that $\gamma\leftrightarrow \gamma'$, then $f(\gamma)=i$,
\item\label{item:col_6} for all $i<j\in[\ell]$, for all $\gamma\in D_{i,j}$, if $\exists\gamma'\not\in D_{i,j}$ such that $\gamma\leftrightarrow\gamma'$, then
\begin{equation}
\begin{cases}
f(\gamma)=i, &\text{if $\gamma'\in D_j$},\\
f(\gamma)=j, &\text{if $\gamma'\in D_i$}.
\end{cases}
\end{equation}
\end{enumerate}
\end{definition}

We comment Definition~\ref{def:partial_coloring}. The first condition says how we color the upper bounds of $\Gamma_Q$; actually, because of~\ref{com:mw_i}, we only need that two antipodals upper bounds have different colors, i.e., $u,u'\in \Upper_Q$ and $u\leftrightarrow u'$ imply $f(u)\neq f(u')$ for all strong $Q$-coloring $f$. We prefer to set $f(u_i)=i$ in order to use color $i,j,k$ instead of $f(u_i),f(u_j),f(u_k)$ and so on; furthermore, condition~\ref{item:col_2} implies the uniqueness of a partial $Q$-coloring. Finally, it is easy to see that conditions~\ref{item:col_5} and~\ref{item:col_6} hold for every strong $Q$-coloring satisfying~\ref{item:col_2}.

Note that there exists always a coloring satisfying conditions~\ref{item:col_2} and~\ref{item:col_5}, while condition~\ref{item:col_6} can not be satisfied if there exist $\gamma\in D_{i,j}$, $\gamma'\in D_i$ and $\gamma''\in D_j$ such that $\gamma\leftrightarrow\gamma'$ and $\gamma\leftrightarrow\gamma''$, for some $i<j\in[\ell]$. This fact leads to a kind of obstruction (see Section~\ref{section:forbidden_subgraphs}).

We are ready to give the definition of weak-coloring.

\begin{definition}\label{def:weak_colorability}
Let $Q$ be a clique separator of $G$. Let $(u_1,u_2,\ldots,u_\ell)$ be any ordering of $\Upper_Q$. 
We say that $G$ is \emph{weak $Q$-colorable} if there exists $f:\Gamma_Q\rightarrow [\ell+1]$ such that $f$ restricted to $\Upper_Q\cup\Ext_Q$ is a partial $Q$-coloring and
\begin{enumerate}[label=\thedefinition.(\arabic*), ref=\thedefinition.(\arabic*)]\itemsep0em
\item\label{item:col_3} for all $i\in[\ell]$, $f(D_i)\subseteq\{i,\ell+1\}$,
\item\label{item:col_4} for all $i<j\in[\ell]$,  $f(D_{i,j})\subseteq\{i,j\}$,
\item\label{item:col_7} for all $D\in\mathcal{D}$, if $\gamma,\gamma'\in D$ and $\gamma\leftrightarrow\gamma'$, then $f(\gamma)\neq f(\gamma')$.
\end{enumerate}
\end{definition}

We refer to a coloring $f$ satisfying the conditions of Definition~\ref{def:weak_colorability} as a \emph{weak $Q$-coloring}.

We comment Definition~\ref{def:weak_colorability}. It is easy to see that if we extend a partial $Q$-coloring, then conditions~\ref{com:mw_i} and~\ref{com:mw_ii} imply conditions~\ref{item:col_4} and~\ref{item:col_7}. The condition~\ref{item:col_3} is more restrictive than the necessary. Indeed, conditions~\ref{com:mw_i} and~\ref{com:mw_ii} should imply $f(D_i)\subseteq\{i,c_i\}$ (a possible choice of $c_i$ is $\ell+i$), but the stiff structure given by the absence of full antipodal triple should imply that all elements colored by $c_i$'s are pairwise not antipodal, and thus we can use the same color for all (as Proposition~\ref{prop:characterization_1_if}'s proof shows). 

In the following theorem we give our first characterization of path graphs. %Its proof is in next subsections.

\begin{theorem}\label{th:characterizazion_1}
A chordal graph $G$ is a path graph if and only if $\Gamma_Q$ has no full antipodal triple and $G$ is weak $Q$-colorable, for all clique separator $Q$ of $G$.
\end{theorem}

\subsection{Weak coloring equals to 2-coloring subproblems}\label{sub:2-coloring}

Now we explain why our coloring problem shown in Theorem~\ref{th:characterizazion_1} simplifies Monma and Wei's one  stated in Theorem~\ref{thm:mw} by an algorithmic point of view. Note that the two conditions of strong $Q$-coloring are in conflict with each other. Indeed, if too few colors are used, then the first condition is violated, otherwise, if one uses too many colors, then the second condition is violated. For this reason the algorithm in~\cite{mew} does not run in polynomial time. In particular, their characterization does not describe directly a polynomial algorithm. Despite this, Sch\"{a}ffer~\cite{schaffer} succeed to implement a sophisticated backtracking polynomial algorithm that starts from their characterization.

Our characterization in Theorem~\ref{th:characterizazion_1} requires the absence of full antipodal triple and the check of six conditions:~\ref{item:col_2}, \ref{item:col_5}, \ref{item:col_6} of Definition~\ref{def:partial_coloring} and~\ref{item:col_3}, \ref{item:col_4}, \ref{item:col_7} of Definition~\ref{def:weak_colorability}. First we note that conditions~\ref{item:col_2}, \ref{item:col_5},  and~\ref{item:col_3},~\ref{item:col_4} are \emph{always} satisfiable, i.e., there exists always a coloring $f:\Gamma_Q\rightarrow [\ell+1]$ that satisfies them. Moreover, checking for the absence of full antipodal triple and condition~\ref{item:col_6} are polynomial problems, because they are antipodal paths/cycles of length 3. Finally, condition~\ref{item:col_7} consists in 2-coloring problems restricted on elements in $D$, for $D\in\mathcal{D}$. In other words, after polynomial checks, we succeed to reduce the coloring problem of Monma and Wei to some 2-coloring subproblems. This allows us to exhibit a list of forbidden subgraphs in the attachedness graph (see Section~\ref{section:forbidden_subgraphs}).

%In the following subsections we prove Theorem~\ref{th:characterizazion_1}. 

\subsection{Proof of Theorem~\ref{th:characterizazion_1}}\label{sub:proof_characterization_1}

This subsection is split into three parts: in Subsection~\ref{sub:preliminary_results} we give some useful results about dominance and antipodality, in Subsection~\ref{sub:strong_implies_weak} and Subsection~\ref{sub:weak_implies_strong} we prove the ``if part'' and the ``only if part'' of Theorem~\ref{th:characterizazion_1}, respectively. In particular, Theorem~\ref{th:characterizazion_1} is implied by Proposition~\ref{prop:characterization_1_if} and by Proposition~\ref{prop:characterization_1_only_if}.

\subsubsection{Preliminary results}\label{sub:preliminary_results}
It is convenient to a have a handy pictorial representation to deal with the relations $\leq$ and $\leftrightarrow$. Two elements $\gamma',\,\gamma''\in \Gamma_Q$ such that $\gamma'\leq \gamma''$ are drawn placing $\gamma'$ below $\gamma''$---here ``below'' means that viewing the sheet as a portion of the Cartesian plane with origin placed in left bottom corner, the ordinate of $\gamma'$ is smaller than the ordinate of $\gamma''$---and joining them by a dotted line while, if $\gamma',\,\gamma''\in \Gamma_Q$ such that $\gamma'\leftrightarrow \gamma''$, then their are joined by a thin line wherever they are placed. For, instance, the following diagrams, represent all possible cases involving three pairwise attached elements of $\Gamma_Q$ (there is not a case that is impossible because of the transitivity of $\leq$).

\begin{equation}\label{diagram:triples}
\begin{gathered}
\xymatrix@R-1pc@C=3mm{
	&  *[o]+<5pt>{\gamma''} &\\
	& & *[o]+<5pt>{\gamma'}\ar@{-}[ul]  \\
	*[o]+<5pt>{\gamma}\ar@{-}[uur]\ar@{-}[urr] & &\\ & (i) &}\qquad
\xymatrix@R-1pc@C=3mm{
	&  *[o]+<5pt>{\gamma''} &\\
	& & *[o]+<5pt>{\gamma'}\ar@{.}[ul]  \\
	*[o]+<5pt>{\gamma}\ar@{.}[uur]\ar@{.}[urr] & &\\ & (ii) &}\qquad	
\xymatrix@R-1pc@C=3mm{
	&  *[o]+<5pt>{\gamma''} &\\
	& & *[o]+<5pt>{\gamma'}\ar@{-}[ul]  \\
	*[o]+<5pt>{\gamma}\ar@{.}[uur]\ar@{-}[urr] & &\\ & (iii) &}\qquad
\xymatrix@R-1pc@C=3mm{
	&  *[o]+<5pt>{\gamma''} &\\
	& & *[o]+<5pt>{\gamma'}\ar@{.}[ul]  \\
	*[o]+<5pt>{\gamma}\ar@{.}[uur]\ar@{-}[urr] & &\\ & (iv) &}\qquad
\xymatrix@R-1pc@C=3mm{
	&  *[o]+<5pt>{\gamma''} &\\
	& & *[o]+<5pt>{\gamma'}\ar@{-}[ul]  \\
	*[o]+<5pt>{\gamma}\ar@{.}[uur]\ar@{.}[urr] & &\\ & (v) &}
\end{gathered}
\end{equation}

\begin{lemma}\label{lemma:elenco_trivial_1}
Let $Q$ be a clique separator of $G$, the following hold:
\begin{enumerate}[label=\thelemma.(\alph*), ref=\thelemma.(\alph*)]\itemsep0em
\item\label{item:neighboring_leq} $\gamma\leq\gamma'\Rightarrow \gamma$ and $\gamma$ are neighboring of $v$, for all $v\in V(\gamma)\cap Q$
\item\label{item:neighboring_ant} $\gamma\leftrightarrow\gamma'\Rightarrow \gamma$ and $\gamma$ are neighboring of $v$, for all $v\in V(\gamma)\cap V(\gamma')\cap Q$,
\item\label{item:neighboring_diagram} let $\gamma,\gamma',\gamma''\in\Gamma_Q$, if one among $(ii)$,$(iii)$,$(iv)$,$(v)$ of \eqref{diagram:triples} applies, then $\gamma,\gamma',\gamma''$ are neighboring of $v$ for all $v\in V(\gamma)\cap V(\gamma')\cap Q$,
\end{enumerate}
\end{lemma}
\begin{proof}
The first two statements follow from definitions of $\leq$ and $\leftrightarrow$. The last statement holds for cases $(ii)$ and $(v)$ by applying~\ref{item:neighboring_leq} on $\gamma\leq\gamma'$ and $\gamma\leq\gamma''$. For cases $(iii)$ and $(iv)$,~\ref{item:neighboring_ant} implies that $\gamma$ and $\gamma'$ are neighboring of $v$ for all $v\in V(\gamma)\cap V(\gamma')\cap Q$, finally,~\ref{item:neighboring_leq} applied to $\gamma$ and $\gamma''$ implies the thesis.
\end{proof}

The following lemma shows that the absence of full antipodal triple gives a stiff structure of antipodality graph with respect to partition $\mathcal{D}$. Indeed, for some $D,D'\in\mathcal{D}$, it holds that $\gamma$ can not be antipodal to $\gamma'$, for $\gamma\in D$ and $\gamma'\in D'$. 

\begin{lemma}\label{lemma:elenco_trivial_2}
Let $Q$ be a clique separator of $G$ and let $i<j\in[\ell]$. If $\Upper_Q$ has no full antipodal triple, then the following hold:
\begin{enumerate}[label=\thelemma(\alph*), ref=\thelemma.(\alph*)]\itemsep0em
\item\label{item:geq_i,j} $\gamma\leq\gamma'$ and $\gamma\in D_{i,j}\Rightarrow\gamma'\in D_{i,j}\cup D_i\cup D_j$,
%\item $\gamma\leftrightarrow\gamma'$ and $\gamma\in D_{i,j}\Rightarrow\gamma'\in D_i\cup D_j\cup D_{i,j}$,
\item\label{item:antipodal_k_k'_1} $\gamma\leftrightarrow\gamma'$ and $\gamma\in D_{i,j}$ $\Rightarrow$ $\gamma'\in D_{i,j}\cup D_i\cup D_j$,
%\item\label{item:antipodal_k_k'_2} $\gamma\leftrightarrow\gamma'$, $\gamma\in D_{i,j}$ and $\gamma'\in D_{k,k'}\Rightarrow$ $k=i$ and $k'=j$,\todob{più chiaro se le metto insieme?}
\item\label{item:antipodal_D_i_Upper} $\gamma\leftrightarrow\gamma'$, $\gamma\in D_i$ and $\gamma'\not\in D_i$ $\Rightarrow\gamma\leftrightarrow u_k$ for $\gamma'\leq u_k$ and $k\neq i$.
\end{enumerate}
\end{lemma}
\begin{proof}
Statements~\ref{item:geq_i,j} and~\ref{item:antipodal_k_k'_1} follow from the absence of full antipodal triple, transitivity of $\leq$,~\ref{item:neighboring_leq} and~\ref{item:neighboring_ant}. Indeed, if one among~\ref{item:geq_i,j} and~\ref{item:antipodal_k_k'_1} is denied, then, by using~\ref{item:neighboring_leq} and~\ref{item:neighboring_ant}, we find a full antipodal triple $\{u_i,u_j,u_k\}\in\Upper_Q$, absurdum by hypothesis. 

To prove~\ref{item:antipodal_D_i_Upper}, we observe that $\gamma'\not\in D_i$ implies that there exists $u_k\in\Upper_Q$ such that $\gamma'\leq u_k$ and $k\neq i$. Such $k$ is unique, indeed, if there exists $k'\neq k$ such that $\gamma'\leq u_{k'}$, then $\{u_k,u_{k'},u_i\}$ is a neighboring set by~\ref{item:neighboring_leq} and~\ref{item:neighboring_ant}, and thus it is full antipodal triple, absurdum by hypothesis. Finally, $\gamma,\gamma',u_i,u_k$ is a neighboring set because of~\ref{item:neighboring_leq} and~\ref{item:neighboring_ant}, and thus $u_i\leftrightarrow u_k$ as claimed.
\end{proof}

Lemma~\ref{lemma:elenco_trivial_2} implies that every partial $Q$-coloring sets the color of $\gamma\in\Gamma_Q$ if and only if $\gamma\in\Ext_Q$. By its definition, the color is univocally determined, i.e., if there exists a partial $Q$-coloring, then it is unique.

\subsubsection{Strong coloring implies weak coloring}\label{sub:strong_implies_weak}

The following lemma shows that a strong $Q$-coloring $f$ can be modified in order to satisfy condition~\ref{item:col_2}, and, if so, then $f$ satisfies conditions~\ref{item:col_5},~\ref{item:col_6},~\ref{item:col_4},~\ref{item:col_7}. This is the first step to prove that definition of strong $Q$-coloring implies the definition of weak $Q$-coloring.

\begin{lemma}\label{lemma:strong_is_partial}
Let $Q$ be a clique separator and let $f:\Gamma_Q\rightarrow[r]$ be a strong $Q$-coloring. Then there exists $g:\Gamma_Q\rightarrow[r']$, with $r'\geq r$, satisfying~\ref{item:col_2}. Moreover, conditions~\ref{item:col_5},~\ref{item:col_6},~\ref{item:col_4},~\ref{item:col_7} hold for every strong $Q$-coloring satisfying~\ref{item:col_2}.
\end{lemma}
\begin{proof}
Let's start with the first part of the claim. By Lemma~\ref{lemma:nofull}, there are no full antipodal triple in $\Gamma_Q$. If $f(u_i)\neq f(u_j)$ for all $i\neq j\in[\ell]$, then the thesis is true. Thus let us assume that there exist $i\neq j\in[\ell]$ such that $f(u_i)=f(u_j)$. We need a preliminary that explains how to obtain a strong $Q$-coloring $g$ satisfying $g(u_i)\neq g(u_j)$ by starting from $f$.
\cadre\label{claim:u_i=u_j}
Let $i,j\in[\ell]$ be such that $f(u_i)=f(u_j)$. Let $\Omega_i=\{\gamma\in\Gamma_Q \,|\, \gamma\leq u_i \text{ and } f(\gamma)=f(u_i)\}$. Let
\begin{equation}\label{eq:r+1}
g(\gamma)=
\begin{cases}
r+1, &\text{if $\gamma\in\Omega_i$},\\
f(\gamma), &\text{otherwise},
\end{cases}
\end{equation}
then $g$ is a strong $Q$-coloring and $g(u_i)\neq g(u_j)$.
\endcadre
\begin{claimproof}{\eqref{claim:u_i=u_j}}
It is clear that $g$ satisfies~\ref{com:mw_i}. Let us assume by contradiction that $g$ does not satisfy~\ref{com:mw_ii}. Thus let $\gamma,\gamma',\gamma''\in\Gamma_Q$ be such that $|g(\{\gamma,\gamma',\gamma''\})|=3$ and $\gamma,\gamma',\gamma''$ are neighboring of $v$ for some $v\in Q$. W.l.o.g., by \eqref{eq:r+1}, we assume that $\gamma\in\Omega_i$, $\gamma',\gamma''\not\in\Omega_i$, $f(\gamma')=f(u_i)$ and $f(\gamma'')\neq f(u_i)$. Indeed, if one of these conditions do not hold, then $|g(\{\gamma,\gamma',\gamma''\})|<3$ because $f$ is a strong $Q$-coloring.

Being $v\in V(\gamma)$ and $\gamma\in\Omega_i$, then $v\in V(u_i)$ by~\ref{item:neighboring_leq}. Being $\gamma'$ neighboring of $v$, then either $\gamma'\leq u_i$ or $u_i\leftrightarrow\gamma'$. If $\gamma'\leq u_i$, then $f(\gamma')=f(u_i)$ implies $\gamma'\in\Omega_i$, absurdum. If $u_i\leftrightarrow\gamma'$, then $f(\gamma')=f(u_i)$ implies that $f$ is not a strong $Q$-coloring, absurdum. Finally, $g(u_i)\neq g(u_j)$ because $g(u_i)=r+1$ while $g(u_j)\leq r$.
\hfill\underline{End proof of \eqref{claim:u_i=u_j}}
\end{claimproof}

By repeatedly applying Claim~\ref{claim:u_i=u_j}, we obtain a strong $Q$-coloring $g$ satisfying~\ref{item:col_2}. To complete the proof, we have to prove that $g$ satifies~\ref{item:col_5},~\ref{item:col_6},~\ref{item:col_4},~\ref{item:col_7}.

Let us assume by contradiction that~\ref{item:col_5} does not hold. Then let $\gamma\in D_i$, for some $i\in[\ell]$, and let $\gamma'\not\in D_i$ be such that $\gamma\leftrightarrow\gamma'$ and assume that $g(\gamma)\neq i$. Being $\gamma'\not\in D_i$, then there exists $u_j\in \Upper_Q$ such that $\gamma'\leq u_j$ and $i\neq j$. By~\ref{item:neighboring_ant}, $\gamma,\gamma'$ are neighboring of $v$, for some $v\in Q$, moreover, by~\ref{item:neighboring_leq}, $u_i,u_j$ are neighboring of $v$. Thus $\gamma,u_i,u_j$ are neighboring of $v$, implying that $\gamma\leftrightarrow u_j$ because $\gamma\in D_i$. Finally,~\ref{com:mw_i} imply $g(\gamma)\neq j$, that implies $|g(\{\gamma,u_i,u_j\}|=3$, absurdum because $g$ is a strong $Q$-coloring.

Let us assume that~\ref{item:col_6} does not hold. Then there exist $\gamma\in D_{i,j}$, $\gamma'\in D_j$ such that $\gamma\leftrightarrow\gamma'$ and $g(\gamma)\neq i$ (the case in which $\gamma'\in D_i$ and $g(\gamma)\neq j$ is similar). As above, by~\ref{item:neighboring_ant} and~\ref{item:neighboring_leq}, it holds that $\gamma,\gamma',u_i,u_j$ are neighboring of $v$ for some $v\in Q$. Moreover, $\gamma\leftrightarrow u_j$ and thus~\ref{item:col_5} implies $g(\gamma')=j$. Being $g$ a strong $Q$-coloring and being $\gamma\leftrightarrow\gamma'$,~\ref{com:mw_i} implies $g(\gamma)\neq j$. Thus $|g(\{\gamma,u_i,u_j\})|=3$, absurdum.

If~\ref{item:col_4} does not hold for a $\gamma\in D_{i,j}$, for some $i<j\in[\ell]$, then $\gamma,u_i,u_j$ would deny~\ref{com:mw_ii}. Indeed, $\gamma,u_i,u_j$ are neighboring of $v$, for every $v\in V(\gamma)\cap Q$, by~\ref{item:neighboring_leq}. Thus $\gamma,u_i,u_j$ deny~\ref{com:mw_ii} because of~\ref{item:col_2}, absurdum.

Finally,~\ref{item:col_7} is implied by~\ref{com:mw_ii}.
\end{proof}

The following proposition is the ``if part'' of the proof of Theorem~\ref{th:characterizazion_1}.

\begin{prop}\label{prop:characterization_1_if}
Let $Q$ be a clique separator of $G$. If $G$ is strong $Q$-colorable, then $\Gamma_Q$ has  no full antipodal triple and $G$ is weak $Q$-colorable.
\end{prop}
\begin{proof}
Let $f:\Gamma_Q\to[r]$ be a strong $Q$-coloring. We have to find $g:\Gamma_Q\to[\ell+1]$ such that $g$ is a weak $Q$-coloring and prove that there are no full antipodal triples in $\Upper_Q$. First we observe that~\ref{com:mw_i} and~\ref{com:mw_ii} implies that there are no full antipodal triples in $\Gamma_Q$.

By Lemma~\ref{lemma:strong_is_partial}, we can assume that $f$ satisfies conditions~\ref{item:col_2},~\ref{item:col_5},~\ref{item:col_6},~\ref{item:col_4},~\ref{item:col_7} and that $r\geq\ell$.

For all $i\in[\ell]$ let $\Omega_i=\{\gamma\in D_i\,|\, f(\gamma)\neq i\}$. Now, let $i\in[\ell]$. For all $\gamma,\gamma'\in\Omega_i$ it holds that $\gamma\not\leftrightarrow\gamma'$, indeed, if not, then $f(\gamma)\neq f(\gamma')$ because of~\ref{com:mw_i}, and this implies $|f(\{\gamma,\gamma',u_i\})|=3$, absurdum because $f$ a strong $Q$-coloring. Thus if we define
\begin{equation*}
g(\gamma)=
\begin{cases}
\ell+i, &\text{if }\gamma\in\Omega_i,\\
f(\gamma), &\text{otherwise,}\\
\end{cases}
\end{equation*}
then $g:\Gamma_Q\to[2\ell]$ is a strong $Q$-coloring because $f$ is a strong $Q$-coloring and we used the same color for a class of non-antipodal elements.

For all distinct $i,j\in[\ell]$, for all $\gamma\in\Omega_i$ and $\gamma'\in\Omega_j$, it holds that $\gamma\not\leftrightarrow\gamma'$. Indeed, if $\gamma\leftrightarrow\gamma'$, then, by assuming $\gamma\in D_i$ and $\gamma'\in D_j$, condition~\ref{item:col_5} implies $f(\gamma)=i$ and $f(\gamma')=j$ and thus $\gamma\not\in\Omega_i$ and $\gamma'\not\in\Omega_j$, absurdum. Finally, let $\Omega=\bigcup_{i\in[\ell]}\Omega_i$ and let
\begin{equation*}
h(\gamma)=
\begin{cases}
\ell+1, &\text{if }\gamma\in\Omega,\\
g(\gamma), &\text{otherwise,}\\
\end{cases}
\end{equation*}
it is clear that $g$ satisfies~\ref{item:col_3}, moreover,~\ref{item:col_3} and~\ref{item:col_4} imply $g:\Gamma_Q\rightarrow [\ell+1]$. By the same above reasoning, $g$ is a strong $Q$-coloring and this finishes the proof.
\end{proof}

\subsubsection{Weak coloring implies strong coloring}\label{sub:weak_implies_strong}

Now we prove the ``only if part'' of Theorem~\ref{th:characterizazion_1}.

\begin{prop}\label{prop:characterization_1_only_if}
Let $Q$ be a clique separator of $G$. If $\Gamma_Q$ has no full antipodal triple and $G$ is weak $Q$-colorable, then $G$ is strong $Q$-colorable.
\end{prop}
\begin{proof}
First we observe that if there is a full antipodal triple, then $\Gamma_Q$ is not strong $Q$-colorable neither weak $Q$-colorable. Thus we assume that there are no full antipodal triples. By Remark~\ref{remark:D_i_is_a_partition}, it holds that $\Gamma_Q=\mathcal{D}$.

Let $f$ be a weak $Q$-coloring of $G$. We have to prove that~\ref{com:mw_i} and~\ref{com:mw_ii} are satisfied. Let $\gamma,\gamma'\in\Gamma_Q$ be such that $\gamma\leftrightarrow\gamma'$. We want to prove that $f(\gamma)\neq f(\gamma')$. Let $i<j\in[\ell]$. Because of the absence of full antipodal triples, then there are four cases: $\gamma,\gamma'\in D_i$ (case 1), $\gamma,\gamma'\in D_{i,j}$ (case 2), $\gamma\in D_i$ and $\gamma'\in D_{i,j}$ (case 3), $\gamma\in D_i$ and $\gamma\in D_j$ (case 4) (there are no other cases because of 
\ref{item:antipodal_k_k'_1}).

If case 1 or case 2 happens, then $f(\gamma)\neq f(\gamma')$, by~\ref{item:col_7}. If case 3 happens, then $f(\gamma)\neq f(\gamma')$, by~\ref{item:col_5} and~\ref{item:col_6}. Finally, if case 4 happens, then $i=f(\gamma)\neq f(\gamma')=j$, by~\ref{item:col_5}. Thus~\ref{com:mw_i} is satisfied.

Now it remains to prove that~\ref{com:mw_ii} holds. Let us assume by contradiction that there exist $\gamma,\gamma',\gamma''\in\Gamma_Q$ such that they form a neighboring triple and $|f(\{\gamma,\gamma',\gamma''\})|=3$. In \eqref{diagram:triples} there are all possibilities of relations between these three elements. Case $(i)$ can not apply because it is a full antipodal triple. Thus we have to prove that in cases $(ii)$,$(iii)$,$(iv)$,$(v)$ $|f(\{\gamma,\gamma',\gamma''\})|<3$ (it is only a fact of checking).

Before we examine every case, we need the following claim.
\cadre\label{claim:strong_weak}
In any case among $(ii)$,$(iii)$,$(iv)$,$(v)$ if there are $\eta\in D_i$ and $\eta'\in D_j$, then $f(\eta)=i$ and $f(\eta')=j$. Similarly, if there are $\mu\in D_{i,j}$ and $\mu'\in D_i$ (resp., $\mu'\in D_j$), then $f(\mu')=i$ (resp., $f(\mu')=j)$.
\endcadre
\begin{claimproof}{\eqref{claim:strong_weak}} 
We note that any case among $(ii)$,$(iii)$,$(iv)$,$(v)$ is a neighboring triple. Thus condition~\ref{item:neighboring_leq} and~\ref{item:neighboring_ant} imply $\eta\leftrightarrow u_j$ and $\eta'\leftrightarrow u_i$, hence $f(\eta)=i$ and $f(\eta')=j$ because of~\ref{item:col_5}. The second part of the claim holds with a similar reasoning.\hfill\underline{End proof of \eqref{claim:strong_weak}}
\end{claimproof}

Let's start with cases $(ii)$ and $(v)$. If $\gamma\in D_i$, for some $i\in[\ell]$, then $\gamma',\gamma''\in D_i$ because of transitivity of $\leq$, thus $f(\{\gamma,\gamma',\gamma''\})\subseteq\{i,\ell+1\}$ by~\ref{item:col_3}. Else, $\gamma\in D_{i,j}$, for some $i<j\in[\ell]$, then $\gamma',\gamma''\in D_i\cup D_j\cup D_{i,j}$ because of~\ref{item:geq_i,j}, thus $f(\{\gamma,\gamma',\gamma''\})\subseteq\{i,j\}$ by~\ref{item:col_3},~\ref{item:col_4} and Claim~\ref{claim:strong_weak}.

Now we deal with case $(iv)$. For short, for any $i\in[\ell]$, we define $\overline{D}_i=(\bigcup_{j<i}D_{j,i})\cup(\bigcup_{j>i}D_{i,j})$. If $\gamma''\in D_{i,j}$, for some $i<j\in[\ell]$, then $\gamma,\gamma'\in D_{i,j}$ because of transitivity of $\leq$, thus $f(\{\gamma,\gamma',\gamma''\})\subseteq\{i,j\}$ by~\ref{item:col_4}. Else, $\gamma''\in D_i$, for some $i\in[\ell]$, then $\gamma,\gamma'\in D_i\cup \overline{D}_{i}$ because of~\ref{item:antipodal_k_k'_1}. There are two sub-cases: either $\gamma,\gamma',\gamma''\in D_i$, or at least one among $\gamma,\gamma'$ is in $\overline{D}_i$. If the first sub-case happens, then we have finished by~\ref{item:col_3}. Otherwise, w.l.o.g., let us assume that $\gamma\in D_{i,j}$ for some $j>i$. Thus $u_j\leftrightarrow\gamma''$, and, by~\ref{item:col_5}, $f(\gamma'')=i$, implying that $f(\{\gamma,\gamma',\gamma''\})\subseteq\{i,j\}$.

It remains to check case $(iii)$. If $\gamma''\in D_{i,j}$, for some $i<j\in[\ell]$, then $\gamma\in D_{i,j}$. By~\ref{item:antipodal_k_k'_1}, $\gamma'\in D_i\cup D_j\cup D_{i,j}$. Thus $f(\{\gamma,\gamma',\gamma''\})\subseteq\{i,j\}$ because of~\ref{item:col_5},~\ref{item:col_6} and~\ref{item:col_7}. Else, $\gamma''\in D_i$, for some $i\in[\ell]$, then there are two sub-cases: either $\gamma\in D_i$, or $\gamma\in D_{i,j}$, for some $i<j\in[\ell]$ (the sub-case $\gamma\in D_{j,i}$ is similar). If the first sub-case happens, then $|f(\{\gamma,\gamma',\gamma''\})|<3$ because $f(\gamma'')=f(\gamma)=i$ by~\ref{item:col_5}. If the second sub-case happens, then $\gamma'\in D_i\cup D_j\cup D_{i,j}$ by above. Thus $f(\{\gamma,\gamma',\gamma''\})\subseteq\{i,j\}$ by~\ref{item:col_3},~\ref{item:col_4} and Claim~\ref{claim:strong_weak}.
\end{proof}

\section{Forbidden subgraphs in attachedness graphs}\label{section:forbidden_subgraphs}

In Subsection~\ref{sub:2-coloring} we described exactly which are the obstructions to the weak coloring, thus we can  list all the obstructions of path graphs in the form of subgraphs of the $Q$-attachedness graphs of a chordal graph $G$, and in Subsection~\ref{sub:LMP} we compare our obstructions with obstructions in~\cite{bfm}. Recall that the $Q$-attachedness graph of $G$ is the graph $(\Gamma_Q,\Join)$ with reflexive pairs neglected---whose edges are therefore distinct pairs $\gamma\gamma'$, $\gamma,\gamma'\in \Gamma_Q$ such that $\gamma\Join\gamma'$. Also recall that the $Q$-antipodality and the $Q$-dominance graph of $G$ factor $(\Gamma_Q,\Join)$. Such a factorization yields a 2-edge coloring of $(\Gamma_Q,\Join)$ which models the interactions between $\leftrightarrow$ and $\leq$. 

In the following definition we present an uncolored version of our obstructions to path graphs.

\begin{definition}
-- For an integer $m$ such that $m\geq 3$, the $m$-\emph{wheel} is the graph on $[m+1]$ where the vertices in $[m]$ induces a cycle and vertex $m+1$ is adjacent to all the other vertices (see Figure~\ref{fig:wheelsandfans}.a).\\
-- For an integer $m$ such that $m\geq 4$, the \emph{$m$-fan} is the graph on $[m]$ such that $[m-1]$ induces a path having end-vertices $1$ and $m-1$ and vertex $m$ is adjacent to all the other vertices (see Figure~\ref{fig:wheelsandfans}.b).\\
-- The \emph{$m$-chorded fan} is the graph obtained from the $m$-fan by adding an edge between vertices $1$ and $m-1$. Notice that the $m$-chorded fan is isomorphic to the $m-1$-wheel (see Figure~\ref{fig:wheelsandfans}.c).\\
-- For an integer $m$ such that $m\geq 4$, the \emph{$m$-double fan} is the graph on $[m]$ such that $[m]$ induces a cycle and vertices $m-1$ and $m$ are adjacent to all other vertices (see Figure~\ref{fig:wheelsandfans}.d).
\end{definition}

\begin{figure}[h]
\captionsetup[subfigure]{justification=centering}
\centering
%FIGURA 1
	\begin{subfigure}{2.6cm}
\begin{overpic}[width=2.6cm,percent]{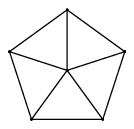}
\end{overpic}
\caption{}\label{fig:ostr_a}
\end{subfigure}
\qquad
%FIGURA 2
	\begin{subfigure}{2.6cm}
\begin{overpic}[width=2.6cm,percent]{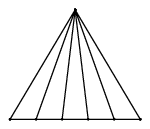}
\end{overpic}
\caption{}\label{fig:ostr_b}
\end{subfigure}
\qquad
%FIGURA 3
	\begin{subfigure}{2.6cm}
\begin{overpic}[width=2.6cm,percent]{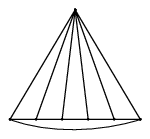}
\end{overpic}
\caption{}\label{fig:ostr_c}
\end{subfigure}	
\qquad
%FIGURA 4
	\begin{subfigure}{2.6cm}
\begin{overpic}[width=2.6cm,percent]{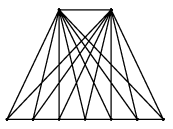}
\end{overpic}
\caption{}\label{fig:ostr_d}
\end{subfigure}	
  \caption{(\subref{fig:ostr_a}) $5$-wheel; (\subref{fig:ostr_b}) $7$-fan; (\subref{fig:ostr_c}) $7$-chored-fan; (\subref{fig:ostr_d}) $9$-double fan.}
\label{fig:wheelsandfans}
\end{figure}

Figure~\ref{fig:ostruzioni_indotte} lists certain special 2-edge-colored graphs, obtained as 2-edge-colored versions of the aforesaid graphs, needed in the characterization of path graphs (Theorem~\ref{cor:all}). The two colors are represented by dotted or solid lines, respectively.

\begin{figure}[h]
\centering
\begin{overpic}[width=10cm]{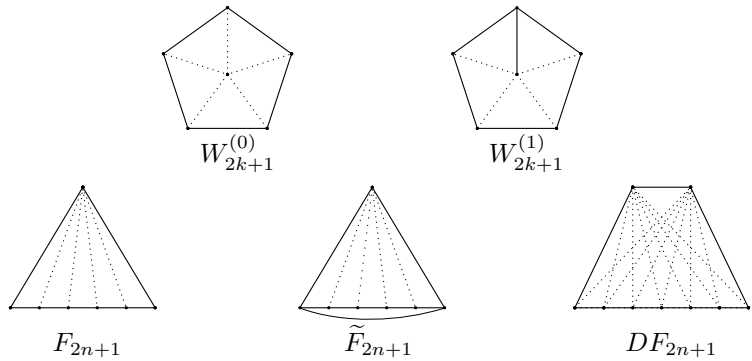} 
\put(26,26){$W^{(0)}_{2k+1}$}    
\put(64,26){$W^{(1)}_{2k+1}$}  
\put(6.7,1){$F_{2n+1}$}  
\put(45,1){$\widetilde{F}_{2n+1}$}  
\put(82,1){$DF_{2n+1}$}         
\end{overpic}
\caption{2-edge-colored graphs occurring in Theorem~\ref{cor:all}, $k\geq1$ and $n\geq2$.}
\label{fig:ostruzioni_indotte}
\end{figure}

It is convenient to settle a specific notation and terminology to present the results. An \emph{isomorphism of edge-colored graphs} is a graph isomorphism which preserves edge colors. All of the 2-edge-colored graphs in Figure~\ref{fig:ostruzioni_indotte} are pairwise non isomorphic as edge-colored graphs. We denote by $\mathcal{F}$ the collection they form ---$\mathcal{F}$ stands for ``forbidden''. Hence
$$\mathcal{F}=\left\{W_{2k+1}^{(0)},\,W_{2k+1}^{(1)},\, F_{2n+1}, \, \widetilde{F}_{2n+1},\, DF_{2n+1} \ |\ k\geq 1,\,n\geq 2\right\}.$$
Also let 
$$\mathcal{F}_0=\left\{W_{2k+1}^{(0)},W_{2k+1}^{(1)},F_{2n+1}\right\}.$$
Triangles of attachedness graphs play a special role. A triangle which is induced by a neighboring triple the $Q$-attachedness graph of $G$ is called a \emph{full triangle}, otherwise it is called \emph{empty}. A triangle all whose edges are antipodal is an \emph{antipodal triangle}. Not every triangle in $Q$-attachedness graph of $G$ is full, indeed an antipodal triangle might be empty (recall the discussion right after Lemma~\ref{lemma:elenco_trivial_1}). 

Unfortunately there is no way to establish whether an antipodal triangle is full or empty, let's see an example. Let $G$ be the graph $F_2$ in Figure~\ref{fig:lmp}, $G$ has only one separator, $Q$, say, let $H$ and $M$ be its $Q$-antipodality and $Q$-attachedness graphs. Hence $M=H\cong K_3$ and the triangle spans a neighboring triple. However, if we denote by $z$ the universal vertex of $G$, then $G-z$  is separated by $Q\setminus z$. Again, let $Q'=Q\setminus z$ be the only clique separator, it holds that $M=H\cong K_3$ but the triangle does not span a neighboring triple.

We know that full antipodal triples are obstructions to strong $Q$-colorability. Therefore, full antipodal triangles are obstructions to membership in the class of path graphs and they should be added to $\mathcal{F}$. However, since full antipodal triangles are not just edge colored triangles (because they have also the property of being full), we must treat such triangles separately in our statements. To overcome this (somehow unaesthetic and noising) ambiguity we use a standard trick.

For a graph $G$ let $G^+$ be the graph defined as follows. Let $V(G)=V=\{v_1,v_2,\ldots,v_n\}$ and $V^+$ be a copy of $V$, $V^+=\{v_1^+,v_2^+,\ldots,v_n^+\}$. Let
\begin{equation*}
	G^+=\left(V\cup V^+,E(G)\cup\{v_iv_i^+\}_{i=1}^n\right).
\end{equation*}

\begin{lemma}\label{lemma:trick} Let $G$ be a graph. Then $G$ is a path graph if and only if $G^+$ is a path graph.
\end{lemma}
\begin{proof}
	Since $G$ is an induced subgraph of $G^+$, $G$ is a path graph if $G^+$ is such.
	Let $T$ be a clique path tree of $G$. For all $v\in V(G)$, let $\mathcal{K}_v$ the set of all cliques of $G$ containing $v$. By Theorem~\ref{thm:gavril}, $\mathcal{K}_v$ induces a path in $T$, let $\tilde{Q}_v\in \mathcal{K}_v$ be an end-vertex of this path. Thus it suffices to join $vv^+$ to $\tilde{Q}_v$ for all $v\in V(G)$ to yield a clique path tree for $G^+$.
\end{proof}
The reason for having introduced graph $G^+$ relies on the fact that, for every clique separator $Q$ of $G^+$, full antipodal triangles of $G$ appear in $Q$-attachedness graph of $G^+$ as small wheels as shown next. 
\begin{lemma}\label{lemma:g+}
	Let $Q$ be a clique separator of $G^+$ and let $M$ be the $Q$-attachedness graph of $G^+$. Then $M$ has no full antipodal triangle and has no induced copy of $W_{2k+1}^{(0)}$if and only if $M$ has no induced copy of $W_{2k+1}^{(0)}$.
\end{lemma}
\begin{proof}
	One direction is trivial. For the other direction it suffices to prove that if $M$ has no induced copy of $W_{3}^{(0)}$, then $M$ has no full antipodal triangle. We prove the contrapositive: if $M$ has a full antipodal triangle, then $M$ has an induced copy of $W_{3}^{(0)}$. Observe first that $Q\cap \{v^+ \ |\ v\in V(G)\}=\emptyset$. For if not, then $Q$ is necessarily of the form $\{v,v^+\}$ for some $v\in V(G)$ (notice that in this case $v$ is a cut vertex); in this case however $M$ would contain no antipodal edges at all and thus no full antipodal triangles. Hence $v^+\not\in Q$ for each $v\in Q$. Notice that for each $v\in Q$ the graph $\gamma^+=(\{v,v^+\},\{vv^+\})$ is $\leq$-dominated by every other neighboring subgraph $\gamma$ of $v$. Let $\{\gamma,\gamma',\gamma''\}$ be the set of vertices of a full antipodal triangle in $M$. Hence, there is some $z\in V(G)$ such that $\gamma$, $\gamma'$ and $\gamma''$ are neighboring subgraphs of $z$. If  $\gamma_z$ is the subgraph of $G$ induced by $\{z,z^+\}\cup Q$, then $\{\gamma_z,\gamma,\gamma',\gamma''\}$ induces a copy of $W_3^{(0)}$ in $M$.
\end{proof}
In the following theorem we claim our characterization by forbidden subgraphs in the attachedness graphs. Note that the graphs in $\mathcal{F}$ are induced obstructions, while the graphs in $\mathcal{F}_0$ are not necessarily
induced. Moreover, statements~\ref{itfin:nof0+} and~\ref{itfin:nof+} are equivalent to~\ref{itfin:nof0} and~\ref{itfin:nof}, respectively, by using $G^+$ in place of $G$ thanks to Lemma~\ref{lemma:trick} and Lemma~\ref{lemma:g+}.

\begin{theorem}\label{cor:all}
Let $G$ be a chordal graph. Then the following statements are equivalent:
	\begin{enumerate}[label=\subscript{\alph*}{\rm )}, ref=\subscript{\alph*}]\itemsep0em
		\item\label{itfin:G_path_graph} $G$ is a path graph,
		\item\label{itfin:nof0} for every clique separator $Q$ of $G$, the $Q$-attachedness graph of $G$ has no full antipodal triangle and has no subgraph isomorphic to any of the graphs in $\mathcal{F}_0$,
		\item\label{itfin:nof0+} for every clique separator $Q$ of $G$, the $Q$-attachedness graph of $G^+$ has no subgraph isomorphic to any of the graphs in $\mathcal{F}_0$,
		\item\label{itfin:nof} for every clique separator $Q$ of $G$, the $Q$-attachedness graph of $G$ has no full antipodal triangle and has no induced subgraph isomorphic to any of the graphs in $\mathcal{F}$,
		\item\label{itfin:nof+} for every clique separator $Q$ of $G$, the $Q$-attachedness graph of $G^+$ has no induced subgraph isomorphic to any of the graphs in $\mathcal{F}$.
	\end{enumerate}
\end{theorem}
The equivalences~\ref{itfin:nof0}$\Leftrightarrow$\ref{itfin:nof0+} and~\ref{itfin:nof}$\Leftrightarrow$\ref{itfin:nof+} in the theorem above follows straightforwardly by Lemma~\ref{lemma:trick} and Lemma~\ref{lemma:g+}. The remaining implications in Theorem~\ref{cor:all} (the core of the characterization), will be the content of the next section.

\subsection{Proof of Theorem~\ref{cor:all}}
We now prove the core of Theorem~\ref{cor:all} according to the schema~\ref{itfin:G_path_graph}$\xLeftrightarrow{\text{Lemma~\ref{lemma:nofan},~\ref{lemma:main}}\,\,}$\ref{itfin:nof0}$\xLeftrightarrow{\text{Lemma~\ref{lemma:ind=sub}}\,\,}$\ref{itfin:nof}; we remember that the equivalences~\ref{itfin:nof0}$\Leftrightarrow$\ref{itfin:nof0+} and~\ref{itfin:nof}$\Leftrightarrow$\ref{itfin:nof+} are implied by Lemma~\ref{lemma:trick} and Lemma~\ref{lemma:g+}. In particular Lemma~\ref{lemma:nofan} implies that every member of $\mathcal{F}_0$ and every full antipodal triangle is an obstruction, Lemma~\ref{lemma:main} explains that $\mathcal{F}_0$ joined with a full antipodal triangle is a minimal set of obstruction and, finally, Lemma~\ref{lemma:ind=sub} shows the equivalence of containing a member of $\mathcal{F}_0$ as subgraphs and a member of $\mathcal{F}$ as induced subgraph. 

In what follows $G$ is chordal graph which is not an atom.

\begin{lemma}\label{lemma:nofan}
If $G$ is a path graph, then, for each clique separator $Q$, the $Q$-attachedness graph of $G$ has neither full antipodal triangles nor copies of any of the graphs in $\mathcal{F}_0$ as subgraphs.
\end{lemma}
\begin{proof}
Let $Q$ be a clique separator. Let us denote by $M$ the $Q$-attachedness graph of $G$. Being $G$ a path graph, then $Q$ contains no full antipodal triangle by Lemma~\ref{lemma:nofull}. Suppose by contradiction that $M$ contains, as a subgraph, a copy $S$ of $F_{2n+1}$ or $W_{2k+1}^{(0)}$ or $W_{2k+1}^{(1)}$. In all cases, $S$ contains a subgraph $F_0$ on $\{\theta_0,\theta_1\dots,\theta_{2t}\}$, $t$ being a positive integer, fulfilling the following conditions.
	\begin{itemize}\itemsep0em
		\item[--] $\theta_i\theta_{i+1}$ in an antipodal edge of $M$, namely $\theta_i\leftrightarrow\theta_{i+1}$, for $i=1,\ldots,2t-1$;
		\item[--] $\theta_0\theta_i$ is a dominance edge of $M$, namely, either $\theta_i\leq\theta_0$ or $\theta_0\leq\theta_i$, for all $i=1,\ldots,2t$.
	\end{itemize} 
We claim that:
	\cadre\label{lemma:scorciatoia}If $f$ is any strong $Q$-coloring of $G$, then $f(\theta_1)\neq f(\theta_{2t})$ and $f(\theta_0)\in\{f(\theta_1),f(\theta_{2t})\}$.
	\endcadre
	\begin{claimproof}{\eqref{lemma:scorciatoia}}
By Lemma~\ref{lemma:elenco_trivial_1} all triangles $\{\theta_0,\theta_i,\theta_{i+1}\}$ are full, for $i=1,\ldots,2t-1$. Hence, being $f$ a strong $Q$-coloring,~\ref{com:mw_ii} implies that $|f(\{\theta_0,\theta_i,\theta_{i+1}\})|=2$, for $i=1,\ldots,2t-1$. Thus if $f(\theta_0)=f(\theta_1)$, then $f(\theta_2)\neq f(\theta_0)$, $f(\theta_3)=f(\theta_0),\dots, f(\theta_{2t})\neq f(\theta_0)$. Instead, if $f(\theta_0)\neq f(\theta_1)$, then $f(\theta_2)= f(\theta_0)$, $f(\theta_3)\neq f(\theta_0),\ldots, f(\theta_{2t})=f(\theta_0)$. In both cases, the thesis follows. \hfill\underline{End proof of \eqref{lemma:scorciatoia}}
\end{claimproof}

We now use Claim \eqref{lemma:scorciatoia} to prove a contradiction to the strong $Q$-colorability of $G$. Suppose first that $S\cong F_{2n+1}$, for some $n$, then let $V(S)=\{\eta,\gamma_1,\dots,\gamma_{2n}\}$ where $\eta$ is the maximum degree vertex of $S$. Let $F'$ be the subgraph induced by $V(S)=\{\eta,\gamma_2,\dots,\gamma_{2n-1}\}$. Hence $F'\cong F_0$. By Claim \eqref{lemma:scorciatoia}, $\gamma_2$ and $\gamma_{2n-1}$  have opposite colors and $f(\gamma_{\eta})\in\{f(\gamma_2),f(\gamma_{2n-1})\}$. Moreover, the triangles induced by $\{\eta,\gamma_1,\gamma_2\}$ and $\{\eta,\gamma_{2n-1},\gamma_{2n}\}$ are both full by Lemma~\ref{lemma:elenco_trivial_1} and at least one of them cannot be 2-colored under $f$. 
	
	Suppose next that $S\cong W_{2k+1}^{(0)}$ or $S\cong W_{2k+1}^{(1)}$ for some $k$. Let $V(S)=\{\eta,\gamma_1,\dots,\gamma_{2k+1}\}$ where $\eta$ is still the maximum degree vertex of $S$ (if $S\cong W_{2k+1}^{(1)}$, then let $\gamma_1$ be the only vertex such that $\gamma_1\eta$ is an antipodal edge) and let $F''$ be the subgraph induced by $V(S)=\{\eta,\gamma_1,\dots,\gamma_{2k}\}$. Clearly, $F''\cong F_0$, as well. By Claim \eqref{lemma:scorciatoia}, $\gamma_1$ and $\gamma_{2k}$  have opposite colors and $f(\eta)\in\{f(\gamma_1),f(\gamma_{2k})\}$. It holds that $f(\gamma_{2k+1})\not\in\{f(\gamma_{2k}),f(\gamma_1)\}$	
 because $\gamma_{2k+1}\leftrightarrow\gamma_{2k}$ and $\gamma_{2k+1}\leftrightarrow\gamma_1$. Moreover, the triangles induced by $\{\eta,\gamma_1,\gamma_{2k+1}\}$ and $\{\eta,\gamma_{2k},\gamma_{2k+1}\}$ are both full by Lemma~\ref{lemma:elenco_trivial_1} and at least one of them cannot be 2-colored under $f$. In any case a contradiction to the strong $Q$-colorability of is achieved.
\end{proof}

\begin{lemma}\label{lemma:main}
If for each clique separator $Q$, the $Q$-attachedness graph of $G$ has neither full antipodal triangles nor copies of any of the graphs in $\mathcal{F}_0$ as subgraphs, then $G$ is path graph.
\end{lemma}
\begin{proof}
By Corollary~\ref{cor:local_mew}, $G$ is a path graph if and only if $G$ is strong $Q$-colorable for each clique separator $Q$. We prove the contrapositive statement, namely, if $G$ is not strong $Q$-colorable for some clique separator $Q$, then the $Q$-attachedness graph $M$ of $G$ contains full antipodal triangles or some copy of a graph of $\mathcal{F}_0$ as subgraphs. Since each graph in $\mathcal{F}$ contains some graph of $\mathcal{F}_0$ as subgraph, we show the statement with $\mathcal{F}$ in place of $\mathcal{F}_0$. Denote by $H$ the $Q$-antipodality graph and remember that $\mathcal{D}$ is a partition of elements of $\Gamma_Q$ if there are not full antipodal triple. For $D\in \mathcal{D}$ denote by $H_D$ the subgraph of $H$ induced by $D$. 

By Theorem~\ref{th:characterizazion_1}, $G$ is not a path graph if one of the following apply: $Q$ contains a full antipodal triple, $Q$ does not admits a partial $Q$-coloring, the $Q$ partial coloring defined on $\Upper_Q\cup\Ext_Q$ can not be extended to a weak $Q$-coloring on $\Gamma_Q$. Let $(u_1,\ldots,u_\ell)$ be any ordering of $\Upper_Q$.

If $Q$ contains a full antipodal triple then this full antipodal triple is also a full antipodal triangle. If $Q$ does not admit a partial $Q$-coloring, then~\ref{item:col_6} is not be satisfiable; indeed~\ref{item:col_2} and~\ref{item:col_5} are always satisfiable. Hence there exist $\gamma\in D_{i,j}$, $\gamma'\in D_i$ and $\gamma''\in D_j$, for some $i<j\in[\ell]$, such that $\gamma\leftrightarrow\gamma'$ and $\gamma\leftrightarrow\gamma''$. Now there are two cases: $\gamma'\leftrightarrow\gamma''$ or $\gamma'\not\leftrightarrow\gamma''$. If the first case applies, then $\{\gamma,\gamma',\gamma'',u_i\}$ induces a copy of $W^{(1)}_3$ in $M$ (refer to Lemma~\ref{lemma:elenco_trivial_1} and Lemma~\ref{lemma:elenco_trivial_2} to determine all colored edges in $M$). Else, the second case applies and $\{\gamma,\gamma',\gamma'',u_i,u_j\}$ induces a copy of $DF_5$.

Thus it remains to study only the case in which $Q$ has no full antipodal triangle, $Q$ admits a partial $Q$-coloring $g:\Upper_Q\cup\Ext_Q\to[\ell]$ and $g$ can not be extended to a weak $Q$-coloring on $\Gamma_Q$. Being~\ref{item:col_3} and~\ref{item:col_4} be always satisfiable by an extension of $g$ (as proved in Lemma~\ref{lemma:strong_is_partial}), then there exists $D\in\mathcal{D}$ such that every extension of $g$ does not satifies~\ref{item:col_7} on $D$. 

Conditions~\ref{item:col_3},~\ref{item:col_4} and~\ref{item:col_7} implies that $H_D$ is 2-colored. Only three cases can occur:
	\begin{itemize}\itemsep0em
		\item[--]  $H_D$ is non bipartite. In this case no 2-coloring $g$ of $H$ exists.
		\item[--] $H_D$ is bipartite but it contains a path $P$ with an even number of vertices whose endvertices have the same color under $g$. 
		\item[--] $H_D$ is bipartite but it contains a path $P$ with an odd number of vertices whose endvertices have different color under $g$. 
	\end{itemize}

In the first case, $H_D$ contains an odd cycle $C$, on $2k+1$ vertices, say, as subgraph. Hence, for $u\in \Upper_Q\cap D$, the subgraph induced by $C\cup \{u\}$ in $H$ contains a copy of $W_{2k+1}^{(0)}$ as a subgraph.
	
In the second case let $\Theta=\{\theta_1,\ldots,\theta_{2k}\}$ be the set of vertices of $P$. Suppose first that $D=D_i$ for some $i\in [\ell]$. By definition of $g$ there are $\gamma,\gamma'\not\in D_i$ such that $\gamma\leftrightarrow\theta_1$ and $\gamma'\leftrightarrow\theta_{2k}$. It holds that $\gamma\leftrightarrow u_i$ and $\gamma'\leftrightarrow u_i$ by the transitivity of $\leq$ and the definition of $D_i$. Now, let $N$ be the subgraph induced by $\Theta\cup\{\gamma,\gamma',u_i\}$. If $\gamma=\gamma'$ then $N$ contains $W^{(1)}_{2k+1}$ as subgraph. If $\gamma\neq\gamma'$, then $N$ contains either $F_{2n+1}$ or $\widetilde{F}_{2n+1}$ according to whether $\gamma\leftrightarrow\gamma'$ or not. 
	If $D=D_{i,j}$, then we obtain the same results with a similar reasoning.
	
The third case can apply only to $D=D_{i,j}$ for some $i,\,j\in [\ell]$, because all the elements of $\Ext_Q\cap D_i$ have the same color $i$ under $g$. Let $\Theta=\{\theta_1,\ldots,\theta_{2k+1}\}$ be the set of vertices of $P$. By the definition of $g$ there are $\gamma\in D_i$ and $\gamma'\in D_j$ such that $\gamma\leftrightarrow\theta_1$ and $\gamma'\leftrightarrow\theta_{2k+1}$. Then $\Theta\cup\{\gamma,\gamma',u_i,u_j\}$ induces a subgraph in $H_{i,j}$ that contains $DF_{2n+1}$ as subgraph.
\end{proof}

\begin{lemma}\label{lemma:ind=sub}
Let $Q$ be a clique separator of $G$. If the $Q$-attachedness graph of $G$ has no full antipodal triangle, then it has a copy of a graph in $\mathcal{F}_0$ as a subgraph if and only if it has a copy of a graph in $\mathcal{F}$ as an induced subgraph.
\end{lemma}
\begin{proof}
Since any graph in $\mathcal{F}_0$ is contained as subgraph in one of the graph in $\mathcal{F}$ one direction is trivial. Let us prove the other direction. Let $H$ and $M$ be the $Q$-antipodality  and $Q$-attachedness graph of $G$. We have to prove that if $M$ contains some copy of a graph of $\mathcal{F}_0$, then $M$ contains an induced copy of some graph of $\mathcal{F}$. Let $S$ be a graph of $\mathcal{F}_0$. For a cycle $C$ of $S$ it is convenient to distinguish between chords that are edges of the antipodality graphs, which we call $a$-chords, from those that are edges of the dominance graph, which we call $d$-chords. 

Let now $C$ be a antipodal odd cycle of $S$ on $2k+1$ vertices for some integer $k\geq2$, i.e., the vertex set of $C$ is $\{\gamma_0,\ldots,\gamma_{2k}\}$ and the edges are $\{\gamma_0\gamma_1,\ldots,\gamma_{2k-1}\gamma_{2k},\gamma_0\gamma_{2k}\}$, where all the edges of $C$ are antipodal edges. Suppose that $C$ has either no $a$-chord, namely $C$ is induced in $H$, or $C$ has precisely the $a$-chord $\gamma_1\gamma_{2k}$. We will show that every graph in $\mathcal{F}_0$ contains such a cycle with possible $d$-chords with an end in $\gamma_0$. The following fact about such a $C$ is crucial to prove the lemma and it implies that if $C$ has at least one $d$-cords with an end in $\gamma_0$, then $C$ induces in $M$ a copy of $F_{2k+1}$, $DF_{2k+1}$ or $\widetilde{F}_{2k+1}$.
\cadre\label{eq:chords}
If $\gamma_0\gamma_j$ is a $d$-chord of $C$ with, say, $\gamma_j\leq \gamma_0$, $j\not\in\{1,2k\}$, then $C$ has $d$-chords $\gamma_0\gamma_l$ with $\gamma_l\leq\gamma_0$, for all $l\not\in\{1,2k\}$. Moreover,
\begin{itemize}\itemsep0em
\item if $C$ is induced in $H$ and $C$ has some other $d$-chord, then $C$ possesses either all $d$-chords $\gamma_1\gamma_j$ with $\gamma_j\leq \gamma_1$, $j\not \in \{0,2\}$, or, symmetrically, all the $d$-chords $\gamma_{2k}\gamma_j$, with $\gamma_j\leq \gamma_{2k}$, $j\not\in\{0,2k-1\}$,
\item if $\gamma_1\gamma_q$ is an $a$-chord of $C$, then $C$ has no other $d$-chords.  
\end{itemize}
\endcadre
\begin{claimproof}{\eqref{eq:chords}} In the first place, observe that $\gamma_{j-1}\leftrightarrow \gamma_j$ and $\gamma_{j+1}\leftrightarrow \gamma_j$ trivially imply $\gamma_{j-1}\Join\gamma_j$ and $\gamma_{j+1}\Join \gamma_j$ hence, by Lemma~\ref{lemma:elenco_trivial_1}, it holds that  $\gamma_0\Join\gamma_{j-1}$ and $\gamma_0\Join\gamma_{j+1}$. Thus $\gamma_0\gamma_{j-1}$ and $\gamma_0\gamma_{j+1}$ are $d$-chords of $C$, because the unique possible $a$-chord is $\gamma_1\gamma_{2k}$. Necessarily $\gamma_{j-1}\leq \gamma_0$ for, if not, then $\gamma_j\leq \gamma_0\leq \gamma_{j-1}$ would imply $\gamma_j\leq \gamma_{j-1}$ contradicting that  $\gamma_{j-1}\leftrightarrow\gamma_j$. By the same reasons, $\gamma_{j+1}\leq \gamma_0$. A repeated application of this argument to $j-1$ and $j+1$ in place of $j$ proves the first part of the claim (see Figure~\ref{fig:eptagons}(\subref{fig:eptagons_a})). 

The first part of the claim is clearly invariant under automorphisms of $C$. Consequently, we deduce that if $C$ has another $d$-chord  $\gamma_h\gamma_\ell$ with $\gamma_\ell\leq \gamma_h$ and $h\not\in\{1,2k\}$, then $C$ has also $d$-chords $\gamma_h\gamma_1$ and $\gamma_h\gamma_{2k}$. But this is impossible because it would imply $\gamma_{2k}\leq \gamma_h\leq \gamma_0$ while we know that $\gamma_0\leftrightarrow\gamma_{2k}$. Hence all the other possible $d$-chords of $C$ have one end in $\{\gamma_1,\gamma_{2k}\}$. On the other hand $C$ cannot possess $d$-chords $\gamma_1\gamma_h$ and $\gamma_{2k}\gamma_\ell$ for some $h,\,\ell\in [2k]$  because, by the first part of the claim, it would possess the $d$-chord $\gamma_1\gamma_{2k}$ and this would imply $\gamma_{2k}\leq \gamma_1$ and $\gamma_1\leq \gamma_{2k}$ and consequently the contradiction $\gamma_1=\gamma_{2k}$ (see Figure~\ref{fig:eptagons}(\subref{fig:eptagons_b}) and Figure~\ref{fig:eptagons}(\subref{fig:eptagons_c})). 
	
It remains to prove that if $\gamma_1\gamma_{2k}$ is an $a$-chord of $C$, then $C$ has no other $d$-chords with one end in $\{\gamma_1,\gamma_{2k}\}$ (hence no other $d$-chords at all, as Figure~\ref{fig:eptagons}(\subref{fig:eptagons_d})) shows). Suppose that $C$ has a $d$-chord with one end in $\{\gamma_1,\gamma_{2k}\}$, $\gamma_1$ say. Then $C$ has the $d$-chord $\gamma_1\gamma_{2k-1}$ by above. Since $\gamma_{2k-1}\leq\gamma_0$, $\gamma_{2k-1}\leq \gamma_1$ and $\gamma_{2k-1}\leftrightarrow \gamma_{2k}$, by Lemma~\ref{lemma:elenco_trivial_1} it follows that  $\{\gamma_0,\gamma_1,\gamma_{2k}\}$ induces a full antipodal triangle in $M$, contradicting that $M$ has no such triangle.\hfill\underline{End proof of \eqref{eq:chords}}  
\end{claimproof}

\begin{figure}[h]
\captionsetup[subfigure]{justification=centering}
\centering
\quad
%FIGURA 1
	\begin{subfigure}{2.5cm}
\begin{overpic}[width=2.5cm,percent]{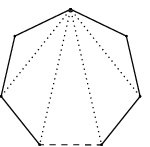}
\put(45,102){$\gamma_0$}
\put(-11,79){$\gamma_{2k}$}
\put(88,79){$\gamma_1$}

\put(100,35){$\gamma_2$}
\put(77,0){$\gamma_3$}
\put(-34,35){$\gamma_{2k-1}$}
\put(-9,0){$\gamma_{2k-2}$}
\end{overpic}
\caption{}\label{fig:eptagons_a}
\end{subfigure}
\qquad\qquad
%FIGURA 2
	\begin{subfigure}{2.5cm}
\begin{overpic}[width=2.5cm,percent]{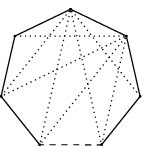}
\put(45,102){$\gamma_0$}
\put(-11,79){$\gamma_{2k}$}
\put(88,79){$\gamma_1$}

\put(100,35){$\gamma_2$}
\put(77,0){$\gamma_3$}
\put(-34,35){$\gamma_{2k-1}$}
\put(-9,0){$\gamma_{2k-2}$}
\end{overpic}
\caption{}\label{fig:eptagons_b}
\end{subfigure}
\qquad\qquad
%FIGURA 3
	\begin{subfigure}{2.5cm}
\begin{overpic}[width=2.5cm,percent]{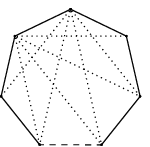}
\put(45,102){$\gamma_0$}
\put(-11,79){$\gamma_{2k}$}
\put(88,79){$\gamma_1$}

\put(100,35){$\gamma_2$}
\put(77,0){$\gamma_3$}
\put(-34,35){$\gamma_{2k-1}$}
\put(-9,0){$\gamma_{2k-2}$}
\end{overpic}
\caption{}\label{fig:eptagons_c}
\end{subfigure}
\qquad\qquad
%FIGURA 4
	\begin{subfigure}{2.5cm}
\begin{overpic}[width=2.5cm,percent]{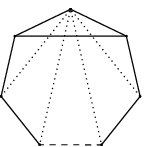}
\put(45,102){$\gamma_0$}
\put(-11,79){$\gamma_{2k}$}
\put(88,79){$\gamma_1$}

\put(100,35){$\gamma_2$}
\put(77,0){$\gamma_3$}
\put(-34,35){$\gamma_{2k-1}$}
\put(-9,0){$\gamma_{2k-2}$}
\end{overpic}
\caption{}\label{fig:eptagons_d}
\end{subfigure}
\caption{graphs in the proof of Claim~\ref{eq:chords}. Note that the graph in (\subref{fig:eptagons_a}) is isomorphic to $F_{2k+1}$, while the graphs in (\subref{fig:eptagons_b}) and (\subref{fig:eptagons_c}) are isomorphic to $DF_{2k+1}$, and the graph in (\subref{fig:eptagons_d}) is isomorphic to $\widetilde{F}_{2k+1}$.}\label{fig:eptagons}
\end{figure}
We can now complete the proof of the lemma. Let $S$ be a copy in $M$ of any of the three graphs in $\mathcal{F}_0$, and let $S$ have $n$ vertices $\gamma_0,\gamma_1,\ldots,\gamma_{n-1}$. Observe that $S$ possesses an odd cycle $R$ on at least $n-1$ vertices; more precisely, the wheels have an odd cycle on $n-1$ vertices and the fan on $n$ vertices. Let $\gamma_0$\ be the highest degree vertex in $S$ and let $H_R$ and $M_R$ be the graphs induced by $R$ in $H$ and $M$, respectively. Let $C$ be a cycle with minimum possible order among the odd cycles of order at least 5 contained in $H_R$. Hence either $C$ is an odd hole of $H$ or $C$ is an odd cycle of $H$ with exactly one $a$-chord which belongs to a triangle having the other two edges on $C$, otherwise the minimality is denied. Clearly, the dominance edges of $S$  induced by $V(C)$ are $d$-chords of $C$. Suppose first that $C$ has no extra $d$-chord other than those. In this case we are done because, $V(C)\cup \{\gamma_0\}$ (possibly $\theta\in V(C)$ when $S$ is a fan) induces either a wheel or a fan or a chorded fan. We may therefore assume that $C$ possesses some extra $d$-chord (a dominance edge of $M_R$ which is not in $S$). Possibly after relabeling, $C$ is of the form  described in Claim~\eqref{eq:chords} and $C$ possesses all the $d$-chords $\gamma_0\gamma_i$, $i\in [n-1]$ (by the claim). If $C$ possesses no other $d$-chords we are done, because $V(C)$ induces either a chorded fan or a fan according to whether or not $C$ possesses the unique $a$-chord $\gamma_1\gamma_t$. If $C$ possesses some other $d$-chord, still by Claim~\eqref{eq:chords}, then $C$ possesses either all $d$-chords $\gamma_1\gamma_j$ with $\gamma_j\leq \gamma_1$, $j\not \in \{0,2\}$, or all the $d$-chords $\gamma_t\gamma_j$, with $\gamma_j\leq \gamma_t$, $j\not\in\{0,t-1\}$. In this case $V(C)$ induces  a double fan in $M$. The proof is completed.
\end{proof}

%The\todob{eliminare frase???} proof of Theorem~\ref{cor:all} is thus completed. Just a final remark: the proof of the equivalence between statements~\ref{itfin:G_path_graph} and~\ref{itfin:nof} obtained via Lemma~\ref{lemma:ind=sub}, may look somehow artificial. This because we have chosen to prove Theorem~\ref{cor:all} by exploiting the characterization of path graph through partial $Q$-colorability (Statement~~\ref{item:st_weak} in Theorem~\ref{th:characterizazion_1}). %Although technically more difficult, it can be proved that the graphs in $\mathcal{F}$ are precisely the induced subgraphs that obstruct the properties expressed in the equivalent Statement~\ref{item:st_loc_part}.

\subsection{Comparison with L\'{e}v\^{e}que, Maffray,~and~Preissmann's characterization}\label{sub:LMP}

We give a brief comparison of our characterization with L\'{e}v\^{e}que, Maffray,~and~Preissmann's characterization~\cite{bfm}, whose list of minimal forbidden subgraphs of the input graph is given in Figure~\ref{fig:lmp}. Table~\ref{table:syn} gives a kind of dictionary between the two characterizations. The table reads as follows. For each row of the table, if a chordal graph $G$ contains an induced copy of one of the subgraphs in the leftmost column (according to L\'{e}v\^{e}que, Maffray,~and~Preissmann's characterization), then each of the graphs in the rightmost column occurs as an induced copy in the $Q$-attachedness graph of $G^+$ for some clique separator $Q$ (according to our characterization). From the table it is apparent a sort of coarsening of the obstructions.

We do not prove how we obtain Table~\ref{table:syn} because it is only a fact of checking, but we report few observations. First of all, it is not necessary to build graph $G^+$ but it suffices to build the attacchedness graph of $G$, for $G$ equal to every obstruction, and observe that a full antipodal triangle corresponds to $W^{(0)}_3$ in the attacchedness graph of $G^+$ (see Lemma~\ref{lemma:g+}'s proof). Obstructions $F_i$ for $i\in\{1,2,3,4,6,7,13,14,15\}$ have exactly one clique separator and thus there is one to one correspondence between L\'{e}v\^{e}que, Maffray,~and~Preissmann's obstructions and ours. Obstructions $F_j$ for $j\in\{8,9,11,16\}$ have exactly two clique separators that are symmetric, thus they generate the same obstruction in $\mathcal{F}$. The same applies on obstructions $F_5(n)$ and $F_{10}(n)$, where the number of clique separators grows with $n$ but all clique separators generate similar attacchedness graphs that have the same obstruction. We have to give particular attention only on $F_{12}(4k)$ because it has two clique separators that generate two different attacchedness graphs, moreover, we need to distinguish the case $k=2$ and the case $k>2$, as we reported in Table~\ref{table:syn}.

Finally we remark that the obstructions in our characterization are 2-edge colored subgraphs and that they have to be forbidden in each graph of the collection of the attachedness graph of $G^+$, while in L\'{e}v\^{e}que, Maffray,~and~Preissmann's characterization the obstructions are forbidden in the input graph itself.

\begin{figure}[h]
\centering
\begin{overpic}[width=14cm]{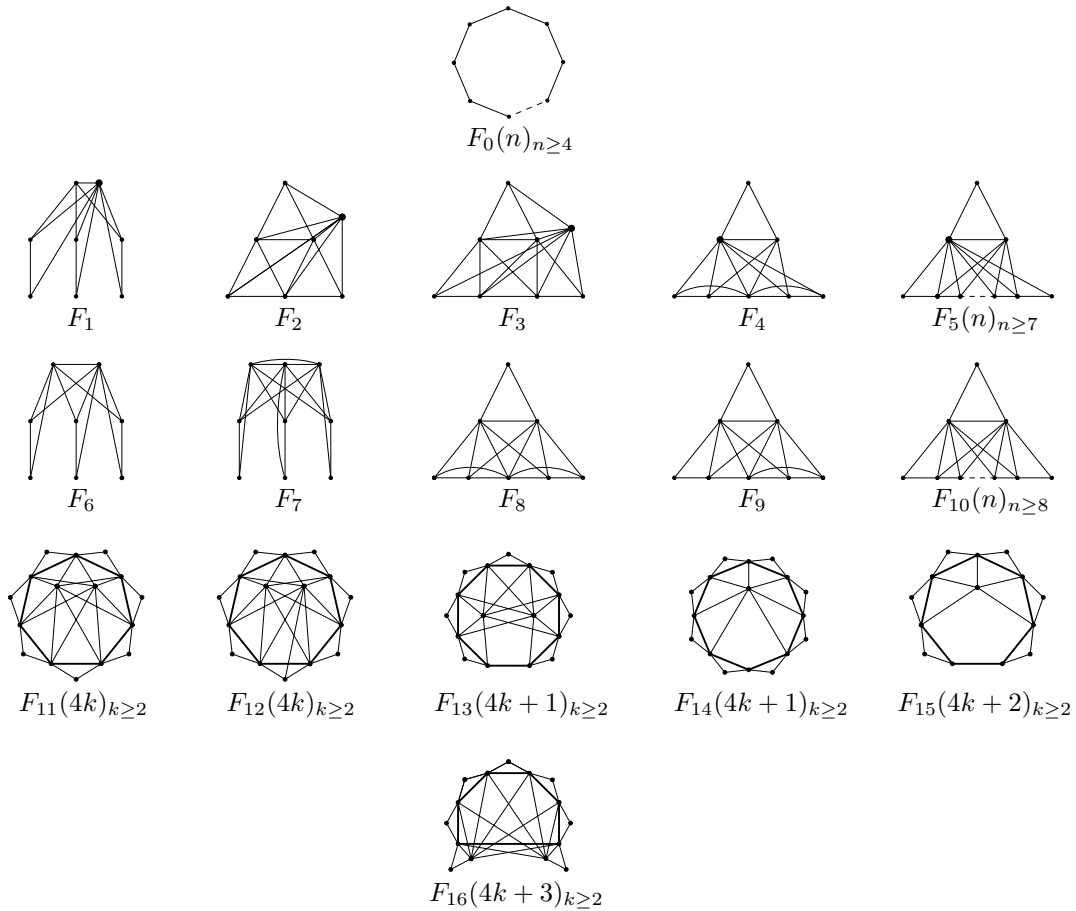}

\put(43.5,74.5){$F_0(n)_{n\geq4}$}

\put(6,57.5){$F_1$}
\put(25.5,57.5){$F_2$}
\put(46.5,57.5){$F_3$}
\put(69,57.5){$F_4$}
\put(87.2,57.5){$F_5(n)_{n\geq7}$}

\put(6,40.3){$F_6$}
\put(25.5,40.3){$F_7$}
\put(46.5,40.3){$F_8$}
\put(69,40.3){$F_9$}
\put(87.2,40.3){$F_{10}(n)_{n\geq8}$}

\put(1.5,21){$F_{11}(4k)_{k\geq2}$}
\put(21,21.){$F_{12}(4k)_{k\geq2}$}
\put(40.5,21.){$F_{13}(4k+1)_{k\geq2}$}
\put(63,21.){$F_{14}(4k+1)_{k\geq2}$}
\put(84,21.){$F_{15}(4k+2)_{k\geq2}$}

\put(40,3){$F_{16}(4k+3)_{k\geq2}$}
\end{overpic}
\caption{L\'{e}v\^{e}que, Maffray and Preissmann's exhaustive list of minimal non path graphs~\cite{bfm} (bold edges form a clique).}
\label{fig:lmp}
\end{figure}

\begin{table}
	\begin{center}
		{\tablinesep=0.01ex\tabcolsep=2pt
			\begin{tabular}{||c c||} 
				\hline
				Family & Obstruction  \\
				\hline\hline
				$F_1, F_2,\ldots, F_9, F_{10}$ & $W_3^{(0)}$ \\ 
				\hline
				$F_{11}(4k)_{k\geq2}$ & $W_{2k-1}^{(0)}$\\
				\hline
				$F_{12}(4k)_{k\geq2}$ & $W_3^{(0)}$, $W_{3}^{(1)}$,  (for $k=2$), $F_{2k-1}$, $W^{(1)}_{2k-1}$ (for $k>2$)\\
				\hline
				$F_{13}(4k+1)_{k\geq2}$ &  $DF_{2k-1}$\\
				\hline
				$F_{14}(4k+1)_{k\geq2}$ &  $\widetilde{F}_{2k+1}$\\
				\hline
				$F_{15}(4k+2)_{k\geq2}$ &  $F_{2k+1}$\\
				\hline
				$F_{16}(4k+3)_{k\geq2}$& $F_{2k+1}$\\
				\hline
				\hline
			\end{tabular}
		}
		\caption{A dictionary between L\'{e}v\^{e}que, Maffray and Preissmann's characterization and Statement~\ref{itfin:nof} in Theorem~\ref{cor:all}. Note that $F_0$ is the obstruction to chordality.}
		\label{table:syn}
	\end{center}
\end{table}

\section{Conclusions}\label{sec:conclusions}

We have presented two new characterizations of path graphs. At the state of the art, our first characterization is the unique that directly describes a polynomial algorithm. In the second one we give a list of local minimal forbidden subgraphs. This paper is the first part of a wider study about path graphs and directed path graphs. The algorithmic consequences are shown in~\cite{balzotti}, in which our first characterization is used to describe a recognition algorithm that specializes for path graphs and directed path graphs. 

We left as open problem the idea of extending our approach to rooted path graphs, for which a list of minimal forbidden subgraphs is unknown, even if some partial results were found~\cite{cameron-hoang_1,gutierrez-leveque,gutierrez-tondato}.

\section*{Acknowledgements}
We wish to thank Paolo G. Franciosa for deep discussion and his helpful advice.

%\bibliographystyle{siam}
%\bibliography{biblio.bib}

\begin{thebibliography}{10}

\bibitem{balzotti}
{\sc L.~Balzotti}, {\em A {N}ew {A}lgorithm to {R}ecognize {P}ath {G}raphs and
  {D}irected {P}ath {G}raphs}, CoRR, abs/2012.08476v2 (2020).

\bibitem{cameron-hoang_1}
{\sc K.~Cameron, C.~T. Ho{\`{a}}ng, and B.~L{\'{e}}v{\^{e}}que}, {\em Asteroids
  in rooted and directed path graphs}, Electron. Notes Discret. Math., 32
  (2009), pp.~67--74.

\bibitem{cameron-hoang_2}
\leavevmode\vrule height 2pt depth -1.6pt width 23pt, {\em Characterizing
  {D}irected {P}ath {G}raphs by {F}orbidden {A}steroids}, J. Graph Theory, 68
  (2011), pp.~103--112.

\bibitem{chaplick}
{\sc S.~Chaplick}, {\em P{Q}{R}-{T}rees and {U}ndirected {P}ath {G}raphs},
  University of Toronto, 2008.

\bibitem{chaplick-gutierrez}
{\sc S.~Chaplick, M.~Gutierrez, B.~L{\'{e}}v{\^{e}}que, and S.~B. Tondato},
  {\em From {P}ath {G}raphs to {D}irected {P}ath {G}raphs}, in Graph Theoretic
  Concepts in Computer Science - 36th International Workshop, {WG} 2010,
  vol.~6410 of Lecture Notes in Computer Science, 2010, pp.~256--265.

\bibitem{dah}
{\sc E.~Dahlhaus and G.~Bailey}, {\em Recognition of {P}ath {G}raphs in
  {L}inear {T}ime}, in 5th Italian Conference on Theoretical Computer Science,
  World Scientific, 1996, pp.~201--210.

\bibitem{gavril1}
{\sc F.~Gavril}, {\em The {I}ntersection {G}raphs of {S}ubtrees in {T}rees are
  {E}xactly the {C}hordal {G}raphs}, J. Combinatorial Theory Ser. B, 16 (1974),
  pp.~47--56.

\bibitem{gavril_DV_algorithm}
\leavevmode\vrule height 2pt depth -1.6pt width 23pt, {\em A {R}ecognition
  {A}lgorithm for the {I}ntersection {G}raphs of {D}irected {P}aths in
  {D}irected {T}rees}, Discret. Math., 13 (1975), pp.~237--249.

\bibitem{gavril_UV_algorithm}
\leavevmode\vrule height 2pt depth -1.6pt width 23pt, {\em A {R}ecognition
  {A}lgorithm for the {I}ntersection {G}raphs of {P}aths in {T}rees}, Discret.
  Math., 23 (1978), pp.~211--227.

\bibitem{gutierrez-leveque}
{\sc M.~Gutierrez, B.~L{\'{e}}v{\^{e}}que, and S.~B. Tondato}, {\em Asteroidal
  quadruples in non rooted path graphs}, Discuss. Math. Graph Theory, 35
  (2015), pp.~603--614.

\bibitem{gutierrez-tondato}
{\sc M.~Gutierrez and S.~B. Tondato}, {\em On models of directed path graphs
  non rooted directed path graphs}, Graphs Comb., 32 (2016), pp.~663--684.

\bibitem{lekkerkerker-boland}
{\sc C.~Lekkekerker and J.~Boland}, {\em Representation of a finite graph by a
  set of intervals on the real line}, Fundamenta Mathematicae, 51 (1962),
  pp.~45--64.

\bibitem{bfm}
{\sc B.~L\'{e}v\^{e}que, F.~Maffray, and M.~Preissmann}, {\em Characterizing
  {P}ath {G}raphs by {F}orbidden {I}nduced {S}ubgraphs}, J. Graph Theory, 62
  (2009), pp.~369--384.

\bibitem{mew}
{\sc C.~Monma and V.~Wei}, {\em Intersection {G}raphs of {P}aths in a {T}ree},
  J. Combin. Theory Ser. B, 41 (1986), pp.~141--181.

\bibitem{mouatadid-robere}
{\sc L.~Mouatadid and R.~Robere}, {\em Path graphs, clique trees, and flowers},
  CoRR, abs/1505.07702 (2015).

\bibitem{panda}
{\sc B.~S. Panda}, {\em The forbidden subgraph characterization of directed
  vertex graphs}, Discret. Math., 196 (1999), pp.~239--256.

\bibitem{renz}
{\sc P.~Renz}, {\em Intersection {R}epresentations of {G}raphs by {A}rcs},
  Pacific J. Math., 34 (1970), pp.~501--510.

\bibitem{schaffer}
{\sc A.~Sch\"{a}ffer}, {\em A {F}aster {A}lgorithm to {R}ecognize {U}ndirected
  {P}ath {G}raphs}, Discrete Appl. Math., 43 (1993), pp.~261--295.

\bibitem{tarjan}
{\sc R.~Tarjan}, {\em Decomposition by {C}lique {S}eparators}, Discrete Math.,
  55 (1985), pp.~221--232.

\end{thebibliography}

\end{document}